\documentclass{amsart} 

\usepackage{algorithm}
\usepackage{kbordermatrix}
\usepackage{amsmath}
\usepackage{amsthm}
\usepackage{amsfonts}
\usepackage{bm} 
\usepackage{mathrsfs} 
\usepackage{multicol}
\usepackage{enumitem}
\usepackage{relsize}
\renewcommand{\int}{\mathlarger{\mathlarger\smallint}}

\usepackage{my-notation}

\theoremstyle{definition}
\newtheorem{definition}{Definition}[section]
\newtheorem{notation}[definition]{Notation}

\newtheorem{example}[definition]{Example}

\newtheorem{problem}[definition]{Problem}

\theoremstyle{plain}
\newtheorem{lemma}[definition]{Lemma}
\newtheorem{proposition}[definition]{Proposition}
\newtheorem{theorem}[definition]{Theorem}

\usepackage{url}
\def\institution{}
\def\streetaddress{}
\def\city{}
\def\country{}
\def\postcode{}
\def\state{}

\begin{document}

\title{Punctual Hilbert Scheme and Certified Approximate Singularities}

\author{Angelos Mantzaflaris, Bernard Mourrain}
\email{angelos.mantzaflaris@inria.fr, bernard.mourrain@inria.fr}

\address{%
  \institution{INRIA Sophia Antipolis, Universit\'e C\^ote d'Azur}
  \streetaddress{2004 route des Lucioles, B.P. 93}
  \city{Sophia Antipolis}
  \country{France}
  \postcode{06902}
}

\author{Agnes Szanto}
\email{aszanto@ncsu.edu}
\address{%
 \institution{Dept. of Mathematics, North~Carolina~State~University}
\streetaddress{Campus Box 8205, Raleigh}
\postcode{27965}
 \city{Raleigh}
 \state{NC}
 \country{USA}
}



\keywords{certification, singularity, multiplicity structure, Newton's method, inverse system, multiplication matrix}

\begin{abstract}
  In this paper we provide a new method to certify that a nearby
  polynomial system has a singular isolated root and we compute its multiplicity structure. More precisely, given a polynomial system
  $\fb=(f_1, \ldots, f_N)\in \K[x_1, \ldots, x_n]^N$, we present a
  Newton iteration on an extended deflated system that locally
  converges, under regularity conditions, to a small deformation of $f$
  such that this deformed system has an exact singular root. The
  iteration simultaneously converges to the coordinates of the
  singular root and the coefficients of the so-called inverse system
  that describes the multiplicity structure at the root. We use
  $\alpha$-theory test to certify the quadratic convergence, and to
  give bounds on the size of the deformation and on the approximation
  error. The approach relies on an analysis of the punctual Hilbert
  scheme, for which we provide a new description. We show in
  particular that some of its strata can be rationally parametrized
  and exploit these parametrizations in the certification. We show
  in numerical experimentation how the approximate inverse system can be 
  computed as a starting point of the Newton iterations and the fast
  numerical convergence to the singular root with its multiplicity
  structure, certified by our criteria. 
\end{abstract}

\maketitle

\section{Introduction}

Local numerical methods such as Newton iterations have proved their efficiency to approximate and certify the existence of simple roots. However for multiple roots they dramatically fail to provide fast numerical convergence and certification.
The motivation for this work is to find a method with fast convergence to an exact singular point and its multiplicity structure for a small perturbation of the input polynomials, and to give numerical tests that can certify it. The knowledge of the multiplicity structure  together with a high precision numerical approximation of a singular solution can be valuable information in many problems.

In \cite{Mourrain97} a method called later {\em integration method} is devised to compute the so-called {\em inverse system} or multiplicity structure at a multiple root. It is used in \cite{mm11} to compute an approximation of the inverse system, given an approximation of that root and to obtain a perturbed system that satisfies the duality property. However, this method did not give a way to improve the accuracy of the initial approximation of the root and the corresponding inverse system. In \cite{Hauensteindeflationmultiplicitystructure2016} a new one-step deflation method is presented that gives an overdetermined polynomial system in the coordinates of the roots and the corresponding inverse system, serving as a starting point for the present paper. However, for certification, \cite{Hauensteindeflationmultiplicitystructure2016} refers to the symbolic-numeric method in \cite{AyyildizAkogluCertifyingsolutionsoverdetermined2018} that only works if the input system is given exactly with rational coefficients and have a multiple root with the prescribed multiplicity structure.

In the present paper we give a solution for the following problem:
%
\begin{problem} \label{prob:main} Given a polynomial system $\fb=(f_1,
  \ldots, f_N)\in \K[\bx]^N$ and a point ${\xi}\in
  \CC^n$, deduce an iterative method that converges quadratically to the triple $(\xi^*, \mu^*, \epsilon^*)$ such that $\xi^*\in \CC^n$, $\mu^*$ defines the coefficients of a basis $\bLambda^*=\{\Lambda^*_1, \ldots, \Lambda^*_\mult\}\subset \CC[\bd_{\xi^*}]$ dual to the set $B_{\xi^*}=\{(\bx-\xi^*)^{\beta_1}, \ldots, (\bx-\xi^*)^{\beta_\mult}\}\subset \CC[\bx]$ and $\epsilon^*$ defines a perturbed polynomial system $\fb_{\epsilon^*}:=\fb + \epsilon^*B_{\xi^*}$ with the property that $\xi^*$ is an exact multiple root of $\fb_{\epsilon^*}$ with inverse system $\bLambda^*$. Furthermore, certify this property and give an upper bound on the size of the perturbation $\|\epsilon^*\|$.
\end{problem}

The difficulty in solving Problem \ref{prob:main} is that known polynomial systems defining the coordinates of the roots and the inverse system are overdetermined, and we need a square subsystem of it in the Newton iterations to guarantee the existence of a root together with the quadratic convergence. Thus, roots of this square subsystem may not be exact roots of the complete polynomial system, and we cannot certify numerically that they are approximations of a root of the complete system. This is the reason why we introduce the variables $\epsilon$ that allow perturbation of the input system. One of the goals of the present paper is to understand what kind of perturbations are needed and to bound their magnitude.

Certifying the correctness of the multiplicity structure that the numerical iterations converge to poses a more significant challenge: the set of parameter values  describing an affine point with multiplicity $\mult$ forms a projective variety called the {\em punctual Hilbert scheme}. The goal is to certify that we converge to a point on this variety. We study an affine subset of the punctual Hilbert scheme and give a new description using multilinear quadratic equations that have a triangular structure. These equations appear in our deflated polynomial system, have integer coefficients, and have to be satisfied exactly without perturbation, otherwise the solution does not define a proper inverse system, closed under derivation. Fortunately, the structure allowed us to define a rational parametrization of a strata of the punctual Hilbert scheme, called the {\em regular} strata. In turn, this rational parametrization allows certification when converging to a point on this regular strata.

Our method comprises three parts: first, we apply the Integration Method (Algorithm 1) with input $\fb$ and ${\xi}$ to compute an approximation of the multiplicity structure, second, an analysis and certification part (see Section 6 and Algorithm 2), and third,  a numerical iteration part converging to the exact multiple root with its multiplicity structure for an explicit perturbation of the input system (see Section 5).


\subsubsection*{Related Work}

There are many works in the literature studying the certification of  isolated singular roots of polynomial systems. One approach is to give  {\em separation bounds} for isolated roots, i.e. a bound that guarantees that there is exactly one root within a neighborhood of  a given point. Worst case separation bounds for square polynomial systems with support in given polytopes and rational coefficients are presented in \cite{emiris:inria-00393833}.
In the presence of singular roots, turned into root clusters after perturbations, these separation bounds separate the clusters from each other and bound the cluster size. \cite{Yakoubsohn2000, Yakoubsohn2002,Giustietal2005} give separation bounds and numerical algorithms to compute clusters of zeroes of univariate polynomials. \cite{DedieuShub2001} extends $\alpha$-theory and gives separation bounds for simple double zeroes of polynomial systems, \cite{GLSY07} extend these results to zeroes of embedding dimension one.

Another approach, called deflation, comprises of transforming the singular root into a regular root of a new system and to apply certification techniques on the new system.
\cite{KanzawaOishi1997} uses a square deflated system to prove the existence of singular solutions.
\cite{lvz06} devises a deflation technique that adds new variables to the systems for isolated singular roots that accelerates Newton's method and \cite{lvz08} modifies this to compute the multiplicity structure.
\cite{RG10} computes error bounds that guarantee the existence of a simple double root within that error bound from the input, \cite{LiZhi2013,LiZhi2014}  generalizes \cite{RG10} to the breadth one case and give an algorithm to compute such error bound. \cite{LiSang2015} gives verified error bounds for isolated and some non-isolated singular roots using higher order deflations. \cite{zeng05,Wu:2008:CMS:1390768.1390812,Zeng2009,WuZhi2011,DaytonLiZeng11,HaoSomZeng2013} give deflation techniques based on numerical linear algebra on the Macaulay matrices that compute the coefficients of the inverse system, with improvements using the closedness property of the dual space.
\cite{GiuYak13,GiustiYak2018} give a new deflation method that does not introduce new variables and extends $\alpha$-theory to general isolated multiple roots for the certification to a simple root of a subsystem of the overdetermined deflated system.
In \cite{Hauensteindeflationmultiplicitystructure2016} a new deflated system is presented, its simple roots correspond to the isolated singular points with their multiplicity structure.
A somewhat different approach is given in
\cite{AyyildizAkogluCertifyingsolutionsoverdetermined2018}, where they use a symbolic-numeric certification techniques that certify that  polynomial systems with rational coefficients have  exact isolated singular roots.
More recently, \cite{LeeLiZhi2019} design a square Newton iteration and provide separation bounds for roots when the deflation method of \cite{lvz06} terminates in one iteration, and give bounds for the size of the clusters.

The certification approach that we propose is based on an algebraic analysis of some strata of the punctual Hilbert scheme. Some of its geometric properties have been investigated long time ago, for instance in \cite{BrianconDescriptionHilb1977, IarrobinoPunctualHilbertschemes1977,BrianconDimensionpunctualHilbert1978} or more recently in the plane \cite{Bejleritangentspacepunctual2017}. However, as far as we know, the effective description that we use and the rational parametrization of the regular strata that we compute have not been developed previously.

\section{Preliminaries}

Let $\fb:= (f_1, \ldots, f_N)\in \Rg^N$ with  $\bx =(x_1, \ldots, x_n)$. Let $\bxi=(\xi_1, \ldots, \xi_n)\in \CC^n$ be an isolated multiple  root of $\fb$.
Let $I=\langle f_1, \ldots, f_N\rangle$, $\m_{\xi}$ be the maximal
ideal at ${\xi}$ and $Q$ be the primary component of $I$ at $\bxi$ so
that $\sqrt{Q}=\m_{\xi}$. The shifted monomials at $\xi$ will
be denoted for $\alpha=(\alpha_{1},\ldots,
\alpha_{n}) \in \NN^{n}$ by
$$\bx_{\xi}^{\alpha}:= (x_{1}-\xi_{1})^{\alpha_{1}}\cdots
(x_{1}-\xi_{n})^{\alpha_{n}}.$$

Consider the ring of power series $\K[[\bdv_\xi]]:= \K[[\dv_{1,\xi}, \ldots, \dv_{n,\xi}]]$ and we denote
$\bdv^{\bbeta}_{\xi}:=\dv_{1, \xi}^{\beta_1}\cdots \dv_{n, \xi}^{\beta_n}$,
with $\bbeta=(\beta_1, \ldots, \beta_n)\in \NN^n$.
We identify $\CC[[\bdv_\xi]]$ with the dual space $\Rg^*$  by
considering the action of $\bdv^{\bbeta}_{\xi}$ on polynomials as derivations and evaluations at $\bxi$,  defined as
\begin{equation}\label{partial}
\bdv^{\bbeta}_{ \xi}(p)
:= \bpartial^{\bbeta}(p)\bigg|_{\bxi}=\frac{\partial^{|\bbeta|  } p}{\partial x_1^{\beta_1}{\cdots} \partial x_n^{\beta_n}} (\bxi) \quad \text{ for } p\in \CC[\bx].
\end{equation}
Hereafter, we reserve the notation $\bdv$ and $\dv_i$ for the dual
variables while $\bpartial$ and $\partial_{x_i}$ for derivation. We
indicate the evaluation at $\xi\in \CC^{n}$ by writing $\dv_{i,\xi}$
and $\bdv_{\xi}$, and  for $\xi=0$ it will be denoted by $\bdv$. The
derivation with respect to the variable $\dv_{i, \xi}$ in $\K[[\bdv_\xi]]$ is denoted
$\partial_{\dv_{i, \xi}}$ $(i=1,\ldots, n)$.
Observe that
$$
 \frac{1}{\bbeta!} \bdv^{\bbeta}_{ \xi}((\bx-\bxi)^\alpha)=\begin{cases} 1  & \text{ if } \alpha=\bbeta, \\
0 & \text{ otherwise,}
\end{cases}
$$
where $\bbeta! = \beta_1!\cdots \beta_n!$.

For $p\in \Rg$ and $ \Lambda\in \K[[\bdv_\xi]]=\Rg^{*}$,
let
$
p\cdot \Lambda: q \mapsto \Lambda (p\,q).
$
We check that $p=(x_i-\xi_i)$ acts as a derivation on
$\CC[[\bdv_\xi]]$: $(x_{i}-\xi_{i}) \cdot \bdv^{\bbeta}_{ \xi}= \partial_{\dv_{i, \xi}} (\bdv^{\bbeta}_{ \xi})=\beta_i \bdv_\xi^{\bbeta-\e_i}$.
Throughout the paper we use the notation $\e_1, \ldots, \e_n$ for the standard basis of $\CC^n$ or for a canonical basis of any vector space $V$ of dimension $n$. We will also use integrals of polynomials in $\CC[[\bdv_\xi]]$ as follows: for $ \Lambda\in \K[[\bdv_\xi]] $ and $k=1, \ldots, n$, $\int_k\Lambda$ denotes the polynomial $\Lambda^* \in \K[[\bdv_\xi]]$ such that $\partial_{\dv_{k, \xi}} (\Lambda^*)=\Lambda$ and $\Lambda^*$ has no constant term. We introduce the following shorthand notation
\begin{eqnarray}\label{eq:shortint}
\tint{k}{\Lambda}:= \int_k\Lambda(\dv_{1, \xi}, \ldots, \dv_{k, \xi}, 0, \ldots, 0).
\end{eqnarray}

For an ideal $I\subset \K[\bx]$, let $I^{\perp}=\{\Lambda \in
\K[[\bdv_{\xi}]]\mid \forall p\in I, \Lambda (p)=0\}$.
The vector space $I^{\perp}$ is naturally identified with the dual
space of $\Rg/I$.
We check that $I^{\perp}$ is a vector subspace of $\K[[\bdv_{\xi}]]$
which is closed under the derivations ${\partial_{\dv_{i, \xi}}}$ for $i=1, \ldots, n$.

\begin{lemma}\label{lem:primcomp}
If $Q$ is a $\m_{\xi}$-primary isolated component of $I$, then $Q^{\perp}=I^{\perp}\cap\K[\bdv_{\xi}]$.
\end{lemma}

This lemma shows that to compute $Q^{\perp}$, it suffices to compute
all polynomials of $\K[\bdv_{\xi}]$ which are in $I^{\perp}$.
Let us denote this set $\DDD= I^{\perp}\cap\K[\bdv_{\xi}]$. It is a
vector space stable under the derivations
$\partial_{\dv_{i, \xi}}$. Its dimension is the dimension of
$Q^{\perp}$ or $\Rg/Q$, that is the {\em multiplicity} of
$\xi$, denoted $\mult_{\xi} (I)$, or simply $\mult$  if $\xi$ and $I$ is clear from the context.

For an element $\Lambda(\bdv_\xi) \in \K[\bdv_\xi]$ we
define the degree or {\em order} $\ord(\Lambda)$ to be the maximal
$|\bbeta|$ s.t. $\bdv^{\bbeta}_{ \xi}$ appears in
$\Lambda(\bdv_\xi)$ with non-zero coefficient.

For $t\in \NN$, let $\DDD_{t}$ be the elements of $\DDD$ of order $\leq t$.
As $\DDD$ is of dimension $\mult$, there exists a smallest $t\geq 0$ s.t.
$\DDD_{t+1}= \DDD_{t}$. Let us call this smallest $t$, the {\em nil-index} of
$\DDD$ and denote it by $\order_{\xi} (I)$, or simply by $\order$. As $\DDD$ is stable by the derivations
$\partial_{d_{i, \xi}}$,
we easily check that for $t\geq \order_{\xi} (I)$, $\DDD_{t}=\DDD$ and
that $\order_{\xi} (I)$ is the maximal degree of elements of $\DDD$.


Let $B=\{\bx_{\xi}^{\bbeta_{1}}, \ldots, \bx_{\xi}^{\bbeta_{\mult}}\}$
be a basis of $\Rg/Q$. We can identify the elements of
$\K[\bx]/Q$ with the elements of the vector space $\sp_\K( B)$.
We define the normal form $N(p)$ of a polynomial $p$ in $\Rg$ as the
unique element $b$ of ${\rm span}_\K(B)$ such
that $p-b\in Q$. Hereafter, we are going to identify the elements of
$\Rg/Q$ with their normal form in $\sp_\K (B)$.
For $\balpha\in \NN^{n}$, we will write the normal form of
$\bx_{\xi}^{\balpha}$ as
\begin{equation}\label{eq:nf}
N(\bx_{\xi}^{\balpha})= \sum_{i=1}^{\mult} \mu_{\bbeta_{i},\balpha}\,
  \bx_{\xi}^{\bbeta_{i}}.
\end{equation}

\subsection{The multiplicity structure}\label{Sec:PointMult}

We start this subsection by recalling the definition of graded
primal-dual pairs of bases for the space $\K[\bx ]/Q$ and its dual.
The following lemma defines the same dual space as in
e.g.~\cite{zeng05,DaytonLiZeng11,LiZhi2014}, but we emphasize on a
primal-dual basis pair to obtain a concrete isomorphism between the
coordinate ring and the dual space.

\begin{lemma}[Graded primal-dual basis pair]\label{pdlemma}
Let $\fb$, $\bxi$,  $Q$, $\DDD$, $\mult= \mult_\xi(\fb)$ and $\order=\order_\xi(\fb)$ be as above.
Then there exists a primal-dual basis pair $(B,\bLambda)$ of the local ring  $\K[\bx ]/ Q$ with the following properties:
\begin{enumerate}
\item The {\em primal basis} of the local ring  $\K[\bx ]/Q$ has the form
\begin{equation}\label{pbasis}
B:=\left\{ \bx_{\xi}^{\bbeta_1},  \bx_{\xi}^{\bbeta_2},\ldots, \bx_{\xi}^{\bbeta_{\mult}}\right\}.
\end{equation}
We can assume that  $\beta_1=0$ and that the ordering of the elements
in $B$ by increasing degree.  Define the set of exponents in $B$ as $\label{E}
E:=\{\bbeta_1, \ldots, \bbeta_{\mult}\}\subset \NN^n$.
\item The unique {\em dual basis} $\bLambda=\{ \Lambda_{1},
  \Lambda_{2},\ldots$, $\Lambda_{{\mult}} \}$ of $\DDD\subset \CC[\bdv_\xi]$
  dual to $B$ has the form
$
\Lambda_{i}=\frac{1}{\bbeta_{i} !}\bdv_\bxi^{\bbeta_{i}}  +\sum_{|\balpha|\leq  |\bbeta_{i}|\atop \balpha\not\in E}\mu_{\bbeta_{i}, \balpha}\; \frac{1}{\bbeta !}\bdv_\bxi^{\balpha}.
$
\item We have
$0=\ord(\Lambda_{1}) \leq \cdots \leq \ord(\Lambda_{\mult})$, and  for  all    $0\leq t\leq \order$ we have
$
\DDD_t=\sp\left\{ \Lambda_{j}\;:\;  \ord(\Lambda_{{j}})\leq t \right\},$ where $\DDD_{t}$ denotes the elements of $\DDD$ of order $\leq t$, as above.
\end{enumerate}
\end{lemma}

A graded primal-dual basis pair $(B,\bLambda)$ of $\DDD$ as described in Lemma \ref{pdlemma} can be obtained from any basis $\tilde{\bLambda}$ of ${\DDD}$ by first choosing pivot elements that are the leading monomials with respect to a graded monomial ordering on $\K[\bdv]$, these leading monomials define $B$, then transforming  the coefficient matrix of $\tilde{\bLambda}$ into row echelon form using the pivot leading coefficients, defining $\bLambda$.

A monomial set $B$ is called a {\em graded primal basis} of $\fb$ at $\xi$
if there exists $\bLambda\subset \CC[\bdv_\xi]$ such that
$(B,\bLambda)$ is a graded primal-dual basis pair and $\bLambda$
is complete for $\fb$ at $\xi$.

Next we describe the so-called {\em integration method} introduced  in \cite{Mourrain97,mm11} that computes a graded pair  of  primal-dual bases as in Lemma \ref{pdlemma} if the root $\xi$ is given.
The integration method performs the computation of a basis  order by
order. We need the following proposition, a new version of
\cite[Theorem 4.2]{Mourrain97}: 

\begin{proposition}\label{prop:integral}
Let $\Lambda_1, \ldots, \Lambda_s\in \CC[\bdv_\xi]$ and assume that $\ord(\Lambda_i)\leq t$ for some $t\in \NN$. Suppose that the subspace $\DDD:={\rm span} (\Lambda_1, \ldots, \Lambda_s)\subset  \CC[\bdv_\xi]$ is closed under derivation.  Then $\Delta\in\CC[\bdv_\xi]$ with no constant term satisfies
$\partial_{\dv_{k}}(\Delta)\in  \DDD$ for all $k=1, \ldots, n$ if and only if $\Delta$ is of the form
\begin{eqnarray}\label{eq:closed}
\Delta =\sum_{i=1}^s\sum_{k=1}^n \nu^k_{i} \tint{k}{ \Lambda_i}
\end{eqnarray}
for some $\nu^k_{i}\in \CC$ satisfying
\begin{eqnarray}\label{eq:commute}
\sum_{i=1}^s\nu^k_{i}\partial_{\dv_{l}}(\Lambda_i)-\nu^l_{i}\partial_{\dv_{k}}(\Lambda_i)=0 \;\text{ for }\;1\leq k<l\leq n.
\end{eqnarray}
Furthermore, (\ref{eq:closed}) and (\ref{eq:commute}) implies that
\begin{eqnarray}\label{eq:deriv}
\partial_{\dv_{k}}(\Delta)=\sum_{i=1}^s \nu^k_{i}\Lambda_i \quad \text{ for } k=1, \ldots, n.
\end{eqnarray}
\end{proposition}
\begin{proof}
  Suppose  $\Lambda\in\CC[\bdv]$ with no constant term satisfies $\partial_{\dv_{k}}(\Lambda)\in  \DDD$ for all $k=1, \ldots,n$. To prove  (\ref{eq:closed}), we can proceed exactly as  in the proof of \cite[Theorem 4.2]{Mourrain97}: we write $\Delta$ uniquely as
$$
\Delta=\Delta_1(d_{1}, \ldots, d_{n})+ \Delta_2(d_{2}, \ldots, d_{n})+\cdots, +\Delta_n( d_{n})
$$
with $\Delta_i\in \CC[d_{i}, \ldots, d_{n}]\setminus \CC[d_{i+1}, \ldots, d_{n}]$. Then $\int_i \partial_{d_{i}}\Delta_i=\Delta_i$. Then we prove that by induction on $k$ that if $\sigma_k:=\Delta_1+\cdots + \Delta_k$ then
$$
\Delta_k=\sum_{j=1}^s\nu_j^k\int _k\Lambda_j - \left(\sigma_{k-1}-\sigma_{k-1}|_{d_k=0}\right)
$$
and
\begin{eqnarray*}
\sigma_k&=&\Delta_k+\sigma_{k-1}=\sum_{j=1}^s\nu_j^k\int _k\Lambda_j + \sigma_{k-1}|_{d_k=0}\\
&=& \sum_{j=1}^s\nu_j^k\int _k\Lambda_j +\sum_{j=1}^s\nu_j^{k-1}\int _k\Lambda_j|_{d_k=0}+\cdots+ \sum_{j=1}^s\nu_j^{1}\int _k\Lambda_j|_{d_k=0,\cdots d_2=0}.
\end{eqnarray*}
Conversely, suppose that $\Lambda\in\CC[\bdv]$ with no constant term is of the form  (\ref{eq:closed}) satisfying (\ref{eq:commute}). Define $\bar{\Delta}_1=\bar{\sigma}_1:= \sum_{j=1}^s\nu_j^1\int _1\Lambda_j $ and for $k=2, \ldots n$ define
$$\bar{\Delta}_k:= \sum_{j=1}^s\nu_j^k\int _k\Lambda_j - \left(\sigma_{k-1}-\sigma_{k-1}|_{d_k=0}\right)$$ and $\bar{\sigma}_k:=\bar{\Delta}_1+\cdots +\bar{\Delta}_k$.
Then  in the proof of \cite[Theorem 4.2]{Mourrain97} it is shown  that $\bar{\Delta}_k \in \CC[d_{k}, \ldots, d_{n}]\setminus \CC[d_{k+1}, \ldots, d_{n}]$ and
$$
\bar{\sigma}_k= \sum_{j=1}^s\nu_j^k\int _k{\Lambda}_j +\sum_{j=1}^s\nu_j^{k-1}\int _k{\Lambda}_j|_{d_k=0}+\cdots+ \sum_{j=1}^s\nu_j^{1}\int _k{\Lambda}_j|_{d_k=0,\cdots d_2=0}
$$
so we get that
$\partial_{d_k}(\Lambda)=\partial_{d_k}(\bar{\sigma}_k)=\sum_{j=1}^s\nu_j^k{\Lambda}_j\in
\DDD_t$ as claimed.
\end{proof}

Let $Q$ be a $\m_{\xi}$-primary ideal. Proposition \ref{prop:integral} implies that if $\bLambda
  =\{\Lambda_{1},\ldots, \Lambda_{r}\}\subset \CC[\bdv_\xi]$ with $\Lambda_{1}=1_\xi$ is a basis of
  $Q^{\perp}$, dual to the basis $B=\{\bx_{\xi}^{\bbeta_{1}},\ldots,
  \bx_{\xi}^{\bbeta_{r}} \}\subset \Rg$ of $\Rg/Q$ with $\ord(\Lambda_{i})=|\bbeta_{i}|$, then there exist $\nu_{i,j}^k\in \CC$ such that
$$
\partial_{d_{k}}(\Lambda_{i})=\sum_{|\bbeta_{j}|<|\bbeta_{i}|}
\nu_{i,j}^{k} \, \Lambda_{j}.
$$ Therefore, the matrix $\mM_{k}$
of the multiplication map $M_k$ by $x_{k}-\xi_{k}$ in the basis $B$ of $\Rg/Q$ is
$$
\mM_{k} = [\nu_{j,i}^{k}]_{1\le i,j\le \mult}^{T} = [\mu_{\bbeta_{i},\bbeta_{j}+\be_{k}}]_{1\le i,j\le \mult}
$$
using the notation \eqref{eq:nf}
and the convention that
$\nu_{i,j}^{k}=\mu_{\bbeta_{i},\bbeta_{j}+\be_{k}}=0$ if
$|\bbeta_{i}|\ge |\bbeta_{j}|$.
Consequently,
$$\nu_{i,j}^{k}=\mu_{\beta_{i}, \beta_{j}+e_{k}} \quad i,j, =1, \ldots, \mult, k=1, \ldots, n,
$$ and we have
$$
\Lambda_{i} = \sum_{|\bbeta_{j}|<|\bbeta_{i}|}\sum_{k=1}^{n} \mu_{\bbeta_{i},\bbeta_{j}+\be_{k}}\tint{k} \Lambda_{j}
$$
where $\mu_{\bbeta_{i},\bbeta_{j} + \be_{k}}$ is the coefficient of
$\bx_{\xi}^{\bbeta_{i}}$ in the normal form of $\bx_{\xi}^{\bbeta_{j}+  \be_{k}}$ in the basis $B$ of $\Rg/Q$.

Next we give a result that allows to simplify the linear systems
involved in the integration method. We first need a definition:
\begin{definition}
Let $E\subset \NN^n$ be a set of exponents. We say that $E$ is {\em closed under division} if $\beta=(\beta_{1}, \ldots, \beta_n) \in E$ implies that $\beta-\e_k\in E$ as long as $\beta_{k}>0$ for all $k=1, \ldots, n$. We also call the corresponding primal basis $B=\{\bx_\xi^{\bbeta_{1}},\ldots,
  \bx_\xi^{\bbeta_{r}} \}$ closed under division.
\end{definition}

The following lemma
provides a simple
characterization of dual bases of inverse systems closed under derivation, that we will use in the integration algorithm.

\begin{lemma}\label{lem:primaldual} Let $B=\{\bx_\xi^{\bbeta_{1}},\ldots,
  \bx_\xi^{\bbeta_{r}} \}\subset \CC[\bx]$ be  closed under
  division and ordered by degree. Let $\bLambda =\{\Lambda_{1},\ldots, \Lambda_{r}\}\subset  \CC[\bd_\xi]$ be a linearly independent set such that
\begin{eqnarray}\label{eq:newLambda}
\Lambda_{i} = \sum_{|\bbeta_{j}|<|\bbeta_{i}|}\sum_{k=1}^{n} \mu_{\bbeta_{i},\bbeta_{j}+\be_{k}}\tint{k} \Lambda_{j}.
\end{eqnarray}
Then $\D=\sp\{\Lambda_{1},\ldots, \Lambda_{r}\}$ is closed under derivation iff  for all $i,s=1, \ldots, r$, $ |\beta_s|<|\beta_i|$ and $k\neq l\in \{1, \ldots, n\}$ we have
\begin{eqnarray}\label{eq:commnew}
\sum_{j: |\beta_s|<|\beta_j|<|\beta_i|}\mu_{\bbeta_i, \bbeta_j+\e_k}\mu_{\beta_j, \beta_s+\e_l}-\mu_{\bbeta_i, \bbeta_j+\e_l}\mu_{\beta_j, \beta_s+\e_k}=0 .
\end{eqnarray}
Furthermore, $(B,\bLambda)$ is a graded primal-dual basis pair iff they satisfy \eqref{eq:commnew} and
\begin{eqnarray}\label{eq:orth}
\mu_{\bbeta_i, \bbeta_j+\e_k}=\begin{cases} 1 & \text{ for } \bbeta_i= \bbeta_j+\e_k\\
0 & \text{ for } \bbeta_j+\e_k \in E, \; \bbeta_i \neq \bbeta_j+\e_k ,
\end{cases}
\end{eqnarray}
\end{lemma}
\begin{proof}[\bf Proof of Lemma \ref{lem:primaldual}]
Assume $\bLambda=\{\Lambda_{1},\ldots, \Lambda_{r}\}$ is linearly
independent and $\D={\rm span}(\Lambda)$ is closed under derivation.
For $t\in \{0, \ldots, \order\}$ denote by $\{\Lambda_1, \ldots,
\Lambda_{r_t}\}=\bLambda \cap  \CC[\bd_\xi]_{t}$ and $\D_{t}={\rm
  span}(\Lambda_1, \ldots, \Lambda_{r_t})$. Then by Proposition
\ref{prop:integral}, $\bLambda$ satisfy equations \eqref{eq:deriv} for
$t=0,  \ldots, \order$ and for $j=1, \ldots, r$, $k=1, \ldots, n$, we have $\partial_{\dv_{k}}(\Lambda_j)=\sum_{|\beta_s|<|\beta_j|} \mu_{\beta_j, \beta_s+\e_k}\Lambda_s$. Substituting this to \eqref{eq:commute} we get for $i=1, \ldots, r$
\begin{eqnarray}\label{eq:intermediate}
\sum_{|\beta_j|<|\beta_i|}\mu_{\bbeta_i, \bbeta_j+\e_k}\sum_{|\beta_s|<|\beta_j|}\mu_{\beta_j, \beta_s+\e_l}\Lambda_s \nonumber\\
\quad -\mu_{\bbeta_i, \bbeta_j+\e_l}\sum_{|\beta_s|<|\beta_j|}\mu_{\beta_j, \beta_s+\e_k}\Lambda_s=0. \;
\end{eqnarray}
Then using linear independence and  collecting the coefficients of $\Lambda_s$ we get \eqref{eq:commnew}.\\
Conversely, assume that \eqref{eq:commnew} is satisfied. Then \eqref{eq:intermediate} is also satisfied. We use induction on $t$ to prove that $\D_t$ is closed under derivation. For $t=0$ there is nothing to prove. Assume  $\D_{t-1}$ is closed under derivation. Then by Proposition \ref{prop:integral} if $|\beta_j|<t$ then $\partial_{\dv_{k}}(\Lambda_j)=\sum_{|\beta_s|<|\beta_j|} \mu_{\beta_j, \beta_s+\e_k}\Lambda_s$ for $ k=1, \ldots, n$. Thus  for $|\beta_i|=t$, \eqref{eq:intermediate} implies  that
$$
\sum_{|\beta_j|<|\beta_i|}\mu_{\bbeta_i, \bbeta_j+\e_k}\partial_{\dv_{l}}(\Lambda_j)-\mu_{\bbeta_i, \bbeta_j+\e_l}\partial_{\dv_{k}}(\Lambda_j)=0.
$$
Again, by Proposition \ref{prop:integral} we get that $\D_t$ is closed under derivation.\\
Next, assume first that $(B,\Lambda)$ is a graded primal-dual basis
pair. This means that for $ i=1,\ldots, \mult$ and for $l$ such that $|\bbeta_l|\leq |\bbeta_i|$
\begin{eqnarray*}
\delta_{i,l}= \Lambda_i\left(\bx_\xi^{\bbeta_l}\right) &=& \sum_{k=1}^n \sum_{|\bbeta_{j}|<|\bbeta_{i}|} \mu_{\bbeta_i, \bbeta_j+\e_k}\tint{k}{\Lambda_j} \left(\bx_\xi^{\bbeta_l}\right)\\
&=& \sum_{k=1}^n \sum_{|\bbeta_{j}|<|\bbeta_{i}|} \mu_{\bbeta_i, \bbeta_j+\e_k}
    {\rm  coeff}({\bdv^{\bbeta_l}\over \bbeta_{l}!}, \;\tint{k}{\Lambda_j})
\end{eqnarray*}
Fix  $k$ to be the index of the last non-zero entry of $\bbeta_l$. For
all other $k$'s $\bdv^{\bbeta_l}$ becomes zero when we
substitute $0$ into $\dv_{k+1},  \ldots, \dv_{n}$ in
$\tint{k}{\Lambda_j}$. Thus,
\begin{eqnarray*}
\Lambda_i\left(\bx_\xi^{\bbeta_l}\right)
&=& \sum_{|\bbeta_{j}|<|\bbeta_{i}|} \mu_{\bbeta_i, \bbeta_j+\e_{k}}
{\rm  coeff}({\bdv^{\bbeta_l}\over \bbeta_{l}!}, \;\tint{k}{\Lambda_j})\\
&=& \sum_{|\bbeta_{j}|<|\bbeta_{i}|} \mu_{\bbeta_i, \bbeta_j+\e_{k}}
    {\rm  coeff}({\bdv^{\bbeta_l-\be_{k}}\over (\bbeta_{l}-\be_{k})!}, \Lambda_j).
\end{eqnarray*}
Since $E$ is closed under division, $\bbeta_l-\e_{k}=\bbeta_m \in E$ for
some $m<l$. By duality,
we have that $ {\rm  coeff}({\bdv^{\bbeta_m}\over (\bbeta_m)!}, \Lambda_j)=\delta_{m,j}$, so
$$
 \Lambda_i\left(\bx_\xi^{\bbeta_l}\right)=\mu_{\beta_i, \beta_m+e_{k}}=\mu_{\beta_i, \beta_l}.
 $$
To satisfy $ \Lambda_i\left(\bx_\xi^{\bbeta_l}\right)=\delta_{i,l}$ we must have
$$\mu_{\bbeta_i, \bbeta_m+\e_k}=\begin{cases} 1 &\text{ if } \bbeta_i=\bbeta_m+\e_k\\
0 & \text{ if }\bbeta_m+\e_k=\bbeta_l\in E\text{ but } i\neq l.
\end{cases}
$$
Conversely, by induction on $t=|\bbeta_{i}|$ we have that
$\deg(\Lambda_{i})\le |\bbeta_{i}|$. Then
$\Lambda_i\left(\bx_{\xi}^{\bbeta_l}\right)=0$ when
$|\bbeta_{l}|>|\bbeta_{i}|$.
For $|\bbeta_{l}|\le |\bbeta_{i}|$, Equations \eqref{eq:orth} imply
that the coefficient of $\bd^{\bbeta_{l}}\over \bbeta_{l}!$ in $\Lambda_{i}$ is $0$ if
$i\neq l$ and $1$ if $i=l$. Therefore  $(B,\Lambda)$ is a graded
primal-dual basis pair.
\end{proof}

To compute the inverse system $\D$ of $\fb$ at a point $\xi$, we will consider the additional systems of equations in
$\xi$ and $\mu=\{\mu_{\beta_{i},\alpha}\}$:
\begin{equation}\label{eq:fperp}
\Lambda_{i}(f_{j})=0\text{ for } 1\le i\le r, 1\le j\le N.
\end{equation}

Throughout the paper we use the following notation:
\begin{notation}\label{not:HtJt}
Let $f_1, \ldots, f_N\in \CC[\bx]$, $\xi\in\CC^{n}$ and fix $t\in \NN$. Let $B_{t-1}=\{\bx_\xi^{\beta_1}, \ldots,$ $\bx_\xi^{\beta_{r_{t-1}}}\}\subset \CC[\bx_\xi]_{t-1}$ be closed under division and  $\bLambda_{t-1}=\{\Lambda_1, \ldots, \Lambda_{r_{t-1}}\}\subset \CC[\bdv_\xi]_{t-1}$ dual to $B_{t-1}$ with $$\partial_{\dv_{k}}(\Lambda_j)=\sum_{|\beta_s|<|\beta_j|} \mu_{\beta_j, \beta_s+\e_k}\Lambda_s\quad j=1,\ldots, r_{t-1}, k=1, \ldots, n.$$ Consider the following homogeneous linear system of equations in the variables $\{\nu_j^k\;:\;j=1, \ldots, r_{t-1}, \; k=1, \ldots, n\}$:
{\small \begin{eqnarray}
\sum_{j: |\beta_s|<|\beta_j|<t}\nu_{j}^{k}\,\mu_{\beta_j,
  \beta_s+\e_l}-\nu_{j}^{l}\, \mu_{\beta_j, \beta_s+\e_k}=0,\quad 1\leq k<l\leq n \label{eq:h1}\label{eq:(1)}\\
\nu_{j}^{k}=0 \quad \textup{ if } \beta_{j}+\e_{k}=\beta_{l} \textup{ for some }1 \le l\le r_{t-1}
   \label{eq:h2}\label{eq:(2)}\quad\quad\quad\quad\quad\quad\quad\\
\left( \sum_{j=1}^{r_{t-1}} \sum_{k=1}^n \nu_j^k\tint{k}{\Lambda_j}\right)\left(f_l\right)=0\quad l=1, \ldots, N.\label{eq:h3}\label{eq:(3)} \quad\quad\quad\quad\quad\quad\quad\quad\quad
\end{eqnarray}}
We will denote by $H_t$ the coefficient matrix of the equations in \eqref{eq:(1)} and \eqref{eq:(2)} and by $K_t$ the coefficient matrix of the equations in \eqref{eq:(1)}-\eqref{eq:(3)}.
\end{notation}

By Proposition \ref{prop:integral} and  Lemma \ref{lem:primaldual}, if $K_{t}\, \nu = 0$ where $\nu=[\nu_{j}^{k}:j=1, \ldots, s, \; k=1, \ldots, n]$, then $\Lambda = \sum_{j=1}^s \sum_{k=1}^n \nu_j^k\tint{k}{\Lambda_j}  \in (\fb)^{\perp} \cap \CC[\bdv_\xi]_{t}=\D_{t}$.
The main loop of the integration method described in Algorithm \ref{AlINT} consists of computing the new basis elements in $\D_{t}$ and the new basis monomials in $B_{t}$ of degree $t$ from the primal-dual basis pair $(B_{t-1}, \bLambda_{t-1})$ in degree $t-1$.

Algorithm~\ref{AlINT} produces incrementally a basis of $\D$, similarly
to Macaulay's method.  The algorithmic advantage is the
smaller matrix size in $O(r\,n^2 +N)$ instead of $N\, {n+\order-1 \choose \order}$,
where $\order$ is the maximal degree (depth) in the dual,
cf.~\cite{mm11,Hauensteindeflationmultiplicitystructure2016}.

\begin{algorithm}[ht]
\caption[AlINT]{\sc Integration Method - Iteration $t$ }
\label{AlINT}
\begin{flushleft}
\noindent{\bf Input:} $t> 0$,  $\fb=(f_1, \ldots , f_N)\in\Rg^N$,
$\xi\in \CC^n$, $B_{t-1}=\{\bx_{\xi}^{\bbeta_1}, \ldots,
\bx_{\xi}^{\bbeta_{r_{t-1}}}\}\subset \CC[\bx]$ closed under division and
$\bLambda_{t-1}=\{\Lambda_1, \ldots, \Lambda_{r_{t-1}}\}\subset
\CC[\bdv_\xi] $ a basis for $\DDD_{t-1}$ dual to $B_{t-1}$, of the form \eqref{eq:newLambda}. \\
{\bf Output:} Either ``$\DDD_t=\DDD_{t-1}$" or $B_t=\{\bx_{\xi}^{\bbeta_1}, \ldots, \bx_{\xi}^{\bbeta_{r_{t}}}\}$ for some $r_{t}>r_{t-1}$ closed under division and
$\bLambda_{t}=\{\Lambda_1, \ldots, \Lambda_{r_{t}}\}$ with $\Lambda_{i}$ of the form \eqref{eq:newLambda},
satisfying \eqref{eq:commnew}, \eqref{eq:orth} and \eqref{eq:fperp}.\\
\hrule\vspace{2mm}
\textbf{(1)} Set up the coefficient matrix $K_t$ of the homogeneous linear system \eqref{eq:h1}-\eqref{eq:h3}
in Notation \ref{not:HtJt}  in the variables $\{\nu_j^k\}_{j=1, \ldots, r_{t-1}, \; k=1, \ldots, n}$ associated to an element of the form $\Lambda= \sum_{j=1}^{r_{t-1}} \sum_{k=1}^n \nu_j^k\tint{k}{\Lambda_j}$. Let $h_t:=\dim\ker K_t$.
\\ \textbf{(2)} If $h_t=0$ then return ``$\DDD_t=\DDD_{t-1}$". If $h_t>0$ define $r_{t}:=r_{t-1}+h_t$.
Perform a triangulation of $H_{t}$ by row reductions with row permutations and column pivoting so that the non-pivoting columns correspond to exponents $\bbeta_{r_{t-1}+1},
 \ldots, \bbeta_{r_{t}}$ with strict divisors in $B_{t-1}$. Let
 $B_{t}= B_{t-1}\cup \{\xb_{\xi}^{\beta_{r_{t-1}+{1}}}, \ldots, \xb_{\xi}^{\beta_{r_{t}}}\}$.
\\ \textbf{(3)} Compute a basis $\Lambda_{r_{t-1}+1}, \ldots, \Lambda_{r_{t}}\in \DRgxi$ of
$\ker K_t$ from the triangular reduction of $H_{t}$ by setting the
coefficients of the non-pivoting columns to $0$ or $1$.
This yields a basis $\bLambda_{t}= \bLambda_{t-1}\cup \{
\Lambda_{r_{t-1}+1}, \ldots, \Lambda_{r_{t}}\}$ dual to $B_{t}$. The coefficients
  $\nu_{i,j}^{k}$ of $\Lambda_{i}$ are 
  $\mu_{\beta_{i},\beta_{j}+\e_{k}}$ in \eqref{eq:newLambda} so that Eq.~\eqref{eq:fperp} are satisfied.
  Eq.~\eqref{eq:orth} are satisfied, since $\bLambda_{t}$ is dual
  to $B_{t}$.
\end{flushleft}
\end{algorithm}

The full {\sc Integration Method}  consists of taking
$\Lambda_1:=1_\bxi$ for $t=0$, a basis of $\D_0$ and then iterating
algorithm {\sc Integration Method - Iteration $t$} until we find a
value of $t$ when $\D_t=\D_{t-1}$. This implies that the order
$\order =\order_\xi(\fb)=t-1$. This leads to the following definition.

\begin{definition}\label{def:compl}
We say that $\bLambda\subset \CC[\bdv_\xi]$ is {\em complete} for $\fb$ at $\xi$ if the linear system $K_{t}$ of the equations
\eqref{eq:h1}-\eqref{eq:h3} in degree $t= \order+1=\ord(\bLambda)+1$
is such that $\ker K_{\order+1}= \{0\}$.
\end{definition}
Notice that the full {\sc Integration Method} constructs a graded primal-dual
basis pair $(B,\bLambda)$.
The basis $\bLambda \subset (\fb)^{\perp}$ spans a space stable
by derivation and is complete for $\fb$, so that we have
${\rm span}({\bLambda})=(\fb)^{\perp}\cap\CC[\bdv_\xi]= Q^{\perp}$
where $Q$ is the primary component of $(\fb)$ at $\xi$.

To guarantee that $B_{t}$ is closed under division, one could
choose a graded monomial ordering $\prec$ of $\CC[\bdv_\xi]$ and
compute an auto-reduced basis of $\ker H_{t}$ such that the initial terms for $\prec$
are $\bd_\xi^{\beta_{i}}$. The set   $B_{t}$ constructed in
this way would be closed under division, since $\D_{t}$ is stable under
derivation. In the approach we use in practice, we
choose the column pivot taking into account the numerical values of
the coefficients and not according to a monomial ordering and we check
a posteriori that the set of exponents is closed under division (See
Example \ref{ex:7.3}).

The main property that we will use for the certification of
multiplicities is given in the next theorem.

\begin{theorem}\label{thm:isolated}
If $\xi^{*}$ is an isolated solution of the system $\fb(\xb)=0$ and $B$
is a graded primal basis at $\xi^{*}$ closed under division, then the
system $F(\xi,\mu)=0$
of all equations \eqref{eq:commnew},
  \eqref{eq:orth} and \eqref{eq:fperp}  admits $(\xi^{*},\mu^{*})$ as an isolated
simple root, where $\mu^{*}$ defines the
basis $\bLambda^{*}$ of the inverse system of $(\fb)$ at $\xi$ dual to $B$, due to \eqref{eq:newLambda}.
\end{theorem}
\begin{proof}
  This is a direct consequence of
  \cite[Theorem 4.11]{Hauensteindeflationmultiplicitystructure2016}, since the system of
  equations \eqref{eq:commnew}-\eqref{eq:fperp}  is
  equivalent to the system (14) in
  \cite[Theorem 4.11]{Hauensteindeflationmultiplicitystructure2016}. The equations
  \eqref{eq:commnew} express the commutation of the transposed of the
  parametric operator of multiplication in $B$, which are the same as
  the equations of commutation of the operators.  By Lemma
  \ref{lem:primaldual}, the equations \eqref{eq:orth} are equivalent to
  the fact that $(B,\bLambda^{*})$ is a graded primal-dual basis
  pair.  Finally, the equations \eqref{eq:fperp} are the same as
  $\N(f_{i})=0$, $i=1,\ldots, s$ where $\N$ is the parametric normal
  form defined in \cite{Hauensteindeflationmultiplicitystructure2016}[see Definition 4.7
  and following remark]. Therefore the two systems are equivalent. By
  \cite[Theorem 4.11]{Hauensteindeflationmultiplicitystructure2016}, they define the
  simple isolated solution $(\xi^{*},\mu^{*})$, where $\mu^{*}$
  defines the basis $\bLambda^{*}$ dual to $B$ due to
  \eqref{eq:newLambda}.
\end{proof}

\section{Punctual Hilbert scheme}\label{sec:Hilb}

The results in Sections \ref{sec:Hilb} and \ref{sec:rational} do not
depend on the point $\xi\in \CC^n$, so to simplify the notation, we
assume in these sections that $\xi={\bf 0}$. Let
$\m=(x_{1},\ldots,x_{n})$ be the maximal ideal defining $\xi={\bf
  0}\in \CC^{n}$.
Let $\DRg$ be the space of polynomials in the variables
$\bdv=(\dv_{1},\ldots,\dv_{n})$ and $\DRg_{t}\subset \DRg$ the subspace of
polynomials in $\d$ of degree $\le t$.

For a vector space $V$, let $\Gr_{\mult}(V)$ be the projective variety
of the $\mult$ dimensional linear subspaces of $V$,
also known as the {\em Grassmannian} of $\mult$-spaces of
$V$.
The points in $\Gr_{\mult}(V)$ are the projective points of
$\PP(\wedge^{\mult}V)$ of the form $\bv = v_{1} \wedge \cdots \wedge
v_{\mult}$ for $v_{i}\in V$.
Fixing a basis $\e_{1}, \ldots, \e_{s}$ of $V$, the Pl\"ucker coordinates
of $\bv$ are the coefficients of $\Delta_{i_{1}, \ldots, i_{\mult}}(\bv)$
of $\bv = \sum_{i_{1}<\cdots<i_{\mult}} \Delta_{i_{1}, \ldots,
  i_{\mult}}(\bv)\, \e_{i_{1}}\wedge \cdots \wedge \e_{i_{\mult}}$.
When $V=\DRg_{\mult-1}$, a natural basis is the dual monomial basis
$({\bdv^{\alpha}\over \alpha !})_{|\alpha|<\mult}$. The Pl\"ucker coordinates
of an element $\bv \in \Gr_{\mult}(\DRg_{\mult-1})$ for this basis are denoted
$\Delta_{\alpha_{1}, \ldots, \alpha_{\mult}}(\bv)$
 where  $\alpha_i\in \NN^n$, $|\alpha_{i}|<\mult$.

If $\bLambda=\{\Lambda_{1}, \ldots, \Lambda_{\mult}\}$ is a basis of a
$\mult$-dimensional space $\D$
in $\DRg_{\mult-1}$ with
$\Lambda_{{i}} = \sum_{|\alpha|<  \mult} \mu_{{i},\alpha}
{\bd^{\alpha}\over \alpha !}$
, the Pl\"ucker coordinates of
$\D$ are, up to a scalar, of the form
$
\Delta_{\alpha_{1}, \ldots,\alpha_{\mult}}= \det \left[
\mu_{i,\alpha_{j}}\right]_{1\le i,j\le \mult}.$
In particular, a monomial set $B=\{\bx^{\beta_{1}},
\ldots, \bx^{\beta_{\mult}}\}\subset \Rg_{\mult-1}$ has a dual basis
in $\D$ iff $\Delta_{\beta_{1}, \ldots, \beta_{\mult}}(\D)\neq 0$.
If $(B=\{\bx^{\beta_{i}}\}_{i=1}^{\mult}, \bLambda=\{\Lambda_{i}\}_{i=1}^\mult)$ is a graded
primal-dual basis pair, then $\mu_{{i},\beta_{j}}= \delta_{i,j}$. To keep our notation consistent with the previous sections, the coordinates of $\Lambda_{i}\in \bLambda$ when $\bLambda$ is dual to $B$ will be denoted by $\mu_{\beta_i, \alpha}$ instead of $\mu_{i,\alpha}$.
By properties of the determinant, the Pl\"ucker coordinates of $\D$ are
such that
\begin{equation}\label{eq:plucker}
\mu_{\beta_{i}, \alpha} = {\Delta_{\beta_{1}, \ldots, \beta_{i-1},
    \alpha, \beta_{i+1}, \ldots, \beta_{\mult}} \over
  \Delta_{\beta_{1}, \ldots, \beta_{\mult}} }\quad i=1, \ldots, \mult.
\end{equation}
If $\D$ is the dual of an ideal $Q= \D^{\perp}\subset \Rg$ and $B=\{\bx^{\beta_{1}},
\ldots, \bx^{\beta_{\mult}}\}$ is a basis of $\Rg/Q$ so that $
\Delta_{\beta_{1}, \ldots, \beta_{\mult}}(\D)\neq 0$, the normal form of $\bx^{\alpha}\in \Rg_{\mult-1}$ modulo
$Q=\D^{\perp}$ in the basis $B$ is
$$
N(\bx^{\alpha}) = \sum_{j=1}^{\mult} \mu_{\beta_j,\alpha}\,
\bx^{\beta_{j}}
= \sum_{j=1}^{\mult} {\Delta_{\beta_{1}, \ldots, \beta_{j-1},
    \alpha, \beta_{j+1}, \ldots, \beta_{\mult}} \over
  \Delta_{\beta_{1}, \ldots, \beta_{\mult}} } \bx^{\beta_{j}}.
$$
(if $\deg(\bx^{\alpha})\ge \mult$, then $N(\bx^{\alpha})=0$).

\begin{definition}
Let $\Hilb_{\mult}\subset \Gr_{\mult}(\DRg_{\mult-1})$ be the set of linear
spaces $\D$ of dimension $\mult$ in $\DRg_{\mult-1}$ which are stable  by the derivations $\partial_{d_{i}}$
with respect to the
variables $\bdv$ (i.e. $\partial_{\dv_{i}}\D \subset \D$ for $i=1,\ldots,n$).
We called $\Hilb_{\mult}$ the {\em punctual Hilbert scheme} of points of multiplicity $\mult$.
\end{definition}

If $\D\subset \DRg$ is stable by the derivations $\partial_{d_{i}}$, then by duality $I=
\D^{\perp} \subset \Rg$ is a vector space of $\Rg$ stable by
multiplication by $x_{i}$, i.e. an ideal of $\Rg$.

\begin{proposition}\label{prop:hilb}
$\D\in \Hilb_{\mult}$ iff $\D^{\perp}=Q$ is an $\m$-primary
ideal such that $\dim \Rg/Q = \mult$.
\end{proposition}
\begin{proof}
Let $\D\in \Hilb_{\mult}$. We prove that $\D^{\perp}=Q$ is an
$\m$-primary ideal. As $\D$ is stable by derivation, $Q=\D^{\perp}$ is
an ideal of $\Rg$. This also implies that $1\in \D$, so that $Q\subset \m$.
As $\dim \D = \dim \Rg/Q = \mult$, $\order=\ord(\D)$ is finite and $\m^{\order+1} \subset
\D^{\perp}=Q$. Therefore, $Q$ is $\m$-primary, which shows the first
implication.

Conversely, let $Q$ be a $\m$-primary
ideal such that $\dim \Rg/Q = \mult$.
Then by Lemma \ref{lem:primcomp},
$\D=Q^{\perp}\subset \DRg_{t}$ is stable by derivation and of
dimension $\mult=\dim \Rg/Q$. Thus $\D\in \Hilb_{\mult}$. This
concludes the proof of the proposition.
\end{proof}

For $\D\in \Hilb_{\mult}$, for $t\geq 0$ we denote by $\D_{t}$ the vector space of
elements of $\D$ of order $\le t$. We verify that $\D_{t}^{\perp}=
\D^{\perp}+ \m^{t+1}$. The next theorem follows from Proposition \ref{prop:integral} and Lemma \ref{lem:primaldual}.

\begin{theorem}{}\label{thm:hilbint}
For $B \subset \Rg$ closed under division such that $|B|=\mult$ and $\order=\deg(B)$, the
following points are equivalent:
\begin{enumerate}
\item $\D\in \Hilb_{\mult}$ and $B_{t}$ is a basis of $\Rg/(\D^{\perp}+ \m^{t+1})$ for $t=1,\ldots, \order$.
\item The dual basis $\bLambda=\{\Lambda_{1}, \ldots, \Lambda_{\mult}\}$ of $B$
  satisfies $\Lambda_{1}=1$ and the equations \eqref{eq:newLambda}, \eqref{eq:commnew} and
  \eqref{eq:orth}.
\end{enumerate}
\end{theorem}

\begin{proof}
  $(1) \Rightarrow (2)$ Assume that $\D\in \Hilb_{\mult}$ and that $B_{t}$ is a
  basis of $\Rg/(\D^{\perp}+ \m^{t+1})$. Let $\bLambda_{t}=\{\Lambda_{1},
  \ldots, \Lambda_{\mult_{t}}\}$ be a basis of $\D_{t}$ dual to $B_{t}$ with $\mult_{t}=|B_{t}|$.
  Then, for $j= \mult_{t-1}+1, \ldots, \mult_{t}$, $\Lambda_{j}\in\D_{t}$  is
  such that $$\partial_{d_{k}}(\Lambda_{j}) = \sum_{j=1}^{\mult_{t-1}}
  \nu_{i,k} \Lambda_{i}$$ for $t=1,\ldots, o$. By Proposition
  \ref{prop:integral}, Equations \eqref{eq:newLambda} and
 \eqref{eq:commnew} are satisfied. As $B_{t}$ is dual to
  $\Lambda_{1}, \ldots, \Lambda_{\mult_{t}}$, Equation
 \eqref{eq:orth} are satisfied.

  $(2) \Rightarrow (1)$ Let $\Lambda_{i}\in \DRg_{\mult-1}$ for  $i=1, \ldots,r$ be
  elements of $\DRg_{\mult-1}$ dual to $B$, which satisfies Equations
 \eqref{eq:newLambda}, \eqref{eq:commnew} and
  \eqref{eq:orth}. By induction on $t=0, \ldots,
  \order=\deg(B)$, we prove that if $\bLambda_{t}=\{\Lambda_{1}, \ldots, \Lambda_{\mult_{t}}\}$ is dual to $B_{t}$,
  then $\Lambda_{1}, \ldots, \Lambda_{\mult_{t}} \in \DRg_{t}$. The
  property is true for $t=0$ since $\Lambda_{1}=1$. If it is true for
  $t-1$, for $\Lambda_{j}$ with $j=\mult_{t-1}+1,\ldots, \mult_{t}$ we
  have by  \eqref{eq:newLambda},  \eqref{eq:commnew}
  and Proposition \ref{prop:integral}, that
 $\partial_{d_{k}}(\Lambda_{j}) = \sum_{j=1}^{\mult_{t-1}} \nu_{i,k}
 \Lambda_{i}$, $k=1,\ldots,n$. Thus $\Lambda_{j}\in \DRg_{t}$.
 This shows that $\D_{t}$ is stable by derivation where
 $\D_{t}\subset \DRg_{t}$ is the vector space spanned $\Lambda_{1}, \ldots,
 \Lambda_{\mult_{t}} \in \DRg_{t}$.
Let $\D= \D_{\order}$. Since, by   \eqref{eq:orth}, $B_{t}$ is dual to $\Lambda_{1}, \ldots, \Lambda_{\mult_{t}} \in
 \DRg_{t}$, we see that $\D \cap \DRg_{t}= \D_{t}$.
 By Proposition \ref{prop:hilb}, $Q=\D^{\perp}$ is a $\m$-primary
 ideal such that $\dim \Rg/Q= \dim \D=|B|=\mult$. Moreover, since
 $B_{t}$ is dual to the basis $\{\Lambda_{1}, \ldots, \Lambda_{\mult_{t}}\}$
 of $\D_{t}$, $B_{t}$ is a basis $\Rg/(\D^{\perp}+ \m^{t+1})$. This
 proves the reverse inclusion.
\end{proof}


For a sequence $\hb=(h_{0},h_{1}, \ldots, h_{\order})\in \NN_{+}^{\order+1}$ and $0\le t\le \order$,
let $\hb_{\vspanh{t}}=(h_{0}, \ldots, h_{t})$,  $r_{t}=
\sum_{i=0}^{t} h_{i}$. For $r\geq 1$
we denote by $\Sb^{r}$ the set of sequences $\hb$ of some length
$\delta< r$ with $h_{i}\neq 0$, $h_{0}=1$ and
$r_{\order}=r$. For $\hb \in \Sb^{r}$, we consider the following subvarieties of $\Hilb_{r_{t}}$:
$$
\Hilb_{\hb_{\vspanh{t}}}=\{\D \in \Hilb_{r_{t}}\mid \dim \D_{i} = \dim \D \cap \DRg_{i}
\le r_{i}, i= 0,\ldots, t\}.
$$
These are projective varieties in $\Hilb_{r_{t}}$ defined by rank
conditions on the linear spaces $\D \cap \DRg_{i}$ for
$\D\in \Hilb_{r_{t}}$, that can be expressed in terms of homogeneous
polynomials in the Pl\"ucker
coordinates of $\D$. In particular, the varieties
$\Hilb_{\hb}:=\Hilb_{\hb_{\vspanh{\order}}}$ are projective subvarieties of
$\Hilb_{r}$. They may not be irreducible or irreducible components of
$\Hilb_{r}$, but we have $\Hilb_{r} ={\cup_{\hb\in \Sb^{r}} \Hilb_{\hb}}$.

We will study a particular component of $\Hilb_{\hb}$, that we
call the {\em regular component of $\Hilb_{\hb}$}, denoted $\Hreg_{\hb}$.
It is characterized as follows. Let $\Hreg_{\hb_{\vspanh{0}}}= \{\vspan{1}\}=\{\DRg_{0}\}= \Gr_{1}(\DRg_{0})$
and assume that $\Hreg_{\hb_{\vspanh{t-1}}}$ has been defined as an
irreducible component of $\Hilb_{\hb_{\vspanh{t-1}}}$. Let
{\small\begin{eqnarray*}
W_{t} = \{ (\D_{t-1}, \E_{t}) \mid \D_{t-1} \in \Hilb_{\hb_{\vspanh{t-1}}},
\E_{t}\in \Gr_{r_{t}}(\DRg_{t}),\\
 \D_{t-1} \subset\E_{t} ,
\forall i\;\partial_{d_{i}} \E_{t} \subset \D_{t-1}  \}
\end{eqnarray*}}
 The constraints $\D_{t-1} \subset\E_{t}$  and $\partial_{d_{i}} \E_{t} \subset \D_{t-1}$ for $i=1, \ldots, n$ define a linear
system of equations in the Pl\"ucker coordinates of $\E_{t}$ (see
e.g. \cite{DoubiletFoundationsCombinatorialTheory1974}),
corresponding to the equations \eqref{eq:closed}, \eqref{eq:commute}.
By construction, the projection of $W_{t}\subset
\Hilb_{\hb_{\vspanh{t-1}}} \times \Gr_{r_{t}}(\DRg_{t})$ on the second factor
$\Gr_{r_{t}}(\DRg_{t})$ is $\pi_{2}(W_{t}) = \Hilb_{\hb_{\vspanh{t}}}$ and
the projection on the first factor is
$\pi_{1}(W_{t})= \Hilb_{\hb_{\vspanh{t-1}}}$.

There exists a dense subset $U_{t-1}$ of the irreducible variety
$\Hreg_{\hb_{\vspanh{t-1}}}$ (with
$\overline{U_{t-1}}=\Hreg_{\hb_{\vspanh{t-1}}}$) such that the rank of
the linear system corresponding to  \eqref{eq:closed} and \eqref{eq:commute} defining $\E_{t}$ is maximal.
Since $\pi_{1}^{-1}(\D_{t-1})$ is irreducible (in fact linear)
of fixed dimension for
$\D_{t-1} \in U_{t-1}\subset \Hreg_{\hb_{\vspanh{t-1}}}$, there is a unique
irreducible component $W_{t,reg}$ of $W_{t}$ such that
$\pi_{1}(W_{t,reg}) = \Hreg_{\hb_{\vspanh{t-1}}}$
(see eg. \cite{ShafarevichBasicalgebraicgeometry2013}[Theorem 1.26]).
We define $\Hreg_{\hb_{\vspanh{t}}} = \pi_{2}(W_{t,reg})$. It is an
irreducible component of $\Hilb_{\hb_{\vspanh{t}}}$, since otherwise
$W_{t,reg}=\pi_{2}^{-1}(\Hreg_{\hb_{\vspanh{t}}})$ would not be
a component of $W_{t}$ but strictly included in one of the
irreducible components of $W_{t}$.

\begin{definition}\label{def:hregt}
 Let $\pi_{t}: \Hilb_{\hb_{\vspanh{t}}}\rightarrow  \Hilb_{\hb_{\vspanh{t-1}}} , \; \D\mapsto
 \D\cap \DRg_{t-1} $ be the projection in
 degree $t-1$.
We define by induction on $t$,
$\Hreg_{\hb_{\vspanh{0}}}=\{\vspan{1}\}$ and $\Hreg_{\hb_{\vspanh{t}}}$
is the irreducible component
$\pi_{t}^{-1}(\Hreg_{\hb_{\vspanh{t-1}}})$ of $\Hilb_{\hb_{\vspanh{t}}}$ for $t=1,\ldots, \order$.
\end{definition}

\section{Rational parametrization}\label{sec:rational}

Let $B=\{\bx^{\beta_{1}}, \ldots, \bx^{\beta_{\mult}}\}\subset
\Rg_{\mult-1}$ be a monomial set. In this section we assume that $B$ is closed under division and its monomials are ordered by increasing degree.
 For $t\in \NN$, we
denote by $B_{t}= B\cap \Rg_{t}$, by $B_{[t]}$ the subset of its monomials
of degree $t$. Let $h_{t}=|B_{[t]}|$, $r_{t}=\sum_{0\le i\le t} h_{t}=
|B_{t}|$ and $\order=\deg(B)$.

Let $\Hilb_{B}:=\{ \D \in \Hilb_{r} \mid B_{t}$ is a basis of
$\Rg/(\D^{\perp}+ \m^{t+1}), t=0,\ldots,\order\}$.
By Theorem \ref{thm:hilbint}, $\Hilb_{B}$ is the set of linear
spaces $\D\in \Hilb_{r}$ such that $\D_{t}= \D \cap \DRg_{t}$ satisfy Equations
\eqref{eq:newLambda} and \eqref{eq:commnew}.
It is the open subset of $\D\in \Hilb_{\hb}$ such that
$\Delta_{B_{t}}(\D_{t})\neq 0$ for $t=1,\ldots, \order$, where $\Delta_{B_{t}}:=\Delta_{\beta_{1}, \ldots, \beta_{r_t}}$ denotes the Pl\"ucker coordinate for $\Gr_{r_t}(\DRg_{t})$ corresponding to the monomials in $B_t$.

Since  for  $\D\in \Hilb_{B}$ we have $\Delta_B(\D)\neq 0$, we can define the affine coordinates of $\Hilb_{B}$ using the coordinates of the elements of the basis $\bLambda=\{\Lambda_1, \ldots, \Lambda_\mult\}$ dual to $B$:
$$
\left\{\mu_{\beta_j,\alpha}= {\Delta_{\beta_{1}, \ldots, \beta_{j-1},
    \alpha, \beta_{j+1}, \ldots, \beta_{\mult}} \over
  \Delta_{B }}\; :\; j=1, \ldots , \mult, \; |\alpha|<r\right\}.
$$
The following lemma shows that  the values of the coordinates
$\{\mu_{\beta_i, \beta_j+\e_k}\; :\; i,j=1, \ldots r, |\beta_j|< |\beta_i|, k=1, \ldots, n\}$ uniquely define $\bLambda$.

\begin{lemma}\label{lem:coord}
Let $B=\{\bx^{\bbeta_1}, \ldots, \bx^{\bbeta_{r_t}}\}$ closed under division, $\D\in \Hilb_{B}$ and $\bLambda=\{\Lambda_1, \ldots, \Lambda_\mult\}$ be the unique basis of $\D$ dual to $B$ with $\Lambda_i=\sum_{|\alpha|\leq |\beta_i|} \mu_{\beta_i, \alpha}\frac{\bd^\alpha}{\alpha!}$ for $i=1, \ldots, \mult$.  Then $\Lambda_1=1$ and for $i=2, \ldots , \mult$
$$
\Lambda_i =\sum_{|\beta_j|<|\beta_i|}\sum_{k=1}^n \mu_{\beta_i, \beta_j+\e_k} \tint{k}{ \Lambda_j}.
$$
 Thus,  $\mu_{\beta_i,\alpha}$  is a polynomial function of $\{\mu_{\beta_s, \beta_j+\e_k}\;:\; |\beta_s|\leq |\beta_i|,  |\beta_j|< |\beta_s|, k=1, \ldots, n\}$ for $i=1, \ldots , \mult, \; |\alpha|<|\beta_i|$.
\end{lemma}

\begin{proof}
Since $\D$ is closed under derivation, by Proposition \ref{prop:integral} there exist $c_{i,s,k}\in \CC$ such that  $\partial_{\dv_{k}}(\Lambda_i)=\sum_{|\beta_s|<|\beta_i|} c_{i,s,k}\Lambda_s$. Then
{\small$$
\mu_{\beta_i, \beta_j+\e_k}= \Lambda_i(\bx ^{\beta_j+\e_k}) =\partial_{\dv_{k}}(\Lambda_i)(\bx ^{\beta_j})= \sum_{|\beta_s|<|\beta_i|} c_{i,s,k}\Lambda_s (\bx ^{\beta_j})=c_{i,j,k}.
$$}
The second claim follows from obtaining the coefficients in $\Lambda$ recursively from $\Lambda_1=1$ and
$$ 
\Lambda_i =\sum_{|\beta_j|<|\beta_i|}\sum_{k=1}^n \mu_{\beta_i, \beta_j+\e_k} \tint{k}{ \Lambda_j}
\textup{ for $i=2, \ldots , \mult$.}
$$
\end{proof}

We define
$\mu:=\{\mu_{\beta_i, \beta_j+\e_k}\}_{i,j=1, \ldots r, |\beta_j|< |\beta_i|, k=1, \ldots, n}$,
$\mu_t:=\{\mu_{\beta_i, \beta_j+\e_k}\in
\mu\;:\;|\beta_i|\leq t \}\subset \mu$ and
$\mu_{[t]}:=\{\mu_{\beta_i,\beta_j+\e_k}\in \mu\;:\; |\beta_j|=
t\}\subset \mu_t$.  The next definition uses the fact  that Equations
\eqref{eq:h1} and  \eqref{eq:h2} are linear in $\nu_{j}^k$ with coefficients depending on $\mu_{t-1}$:

\begin{definition}\label{def:H}
Given $\D_{t-1}\in \Hilb_{B_{t-1}}$ with a unique basis
$\bLambda_{t-1}=\{\Lambda_1, \ldots, \Lambda_{r_{t-1}}\}$ with
$\Lambda_i=\sum_{|\alpha|< t} \mu_{\beta_i,
  \alpha}\frac{\bd^\alpha}{\alpha!}$ for $j=1, \ldots, r_{t-1}$ that
is dual to $B_{t-1}$, uniquely determined by
$\mu_{t-1}=\{\mu_{\beta_i, \beta_j+\e_k}:\;|\beta_i|\leq t-1,
|\beta_j|< |\beta_i|  \}$ as above. Recall from Notation \ref{not:HtJt} that  $H_t$ is the coefficient matrix of the homogeneous linear system \eqref{eq:h1} and \eqref{eq:h2} in the variables $\{\nu_j^k\;:\; j=1, \ldots, r_{t-1}, \; k=1, \ldots, n\}$. To emphasize the dependence of its coefficients on $\D_{t-1}$ or $\mu_{t-1}$ we use the notation   $H_t(\D_{t-1})$ or  $H_t(\mu_{t-1})$.   For $\D\in \Hilb^{reg}_{\hb}$ in an open subset, the rank $\rho_{t}$ of
$H_{t}(\D_{t-1})$ is maximal. 
\end{definition}

The next definition describes a property of a monomial set $B$ such that it will allow us to give a rational parametrization of $\Hilb_{B}$.

\begin{definition}\label{def:non-deg}
 For $t=1,\ldots,\order=\deg(B)$ we say that $\D_t\in \Gr_{r_t}(\DRg_{t})$ is regular for $B_t$
if,
\begin{itemize}[leftmargin=*,labelsep=.3cm,align=left]
   \item $\dim(\D_{t}) = r_{t} = |B_{t}|$,
   \item $\textup{rank}\, H_{t}(\D_{t-1})=\rho_{t}$ the generic rank
     of $H_{t}$ on $\Hilb_{\hb_{t}}^{reg}$,
   \item $\Delta_{B_{[t]}}(\D_{[t]}) \neq 0$ where
     $\Delta_{B_{[t]}}(\D_{[t]})$ is the Pl\"ucker coordinate of $\D_{[t]}\in \Gr_{h_{t}}(\DRg_{r})$ corresponding to the monomials in
     $B_{[t]}$.
\end{itemize}
Let $U_{t}:=\{\D_{t}\in \Hilb^{reg}_{\hb_{t}}\;:\; \D_t \text{ is regular for } B_t\}$.
Then  $U_{t}$ is either an open dense subset of the
irreducible variety $\Hilb^{reg}_{\hb_{t}}$ or empty if
$\Delta_{B_{[t]}}(\D_{[t]})=0$ for all $\D\in \Hilb^{reg}_{\hb_{t}}$. 
We say that $B$ is a {\em  regular basis} if
$\overline{U_{t}}=\Hilb^{reg}_{\hb_{t}}$ (or $U_{t}\neq \emptyset$)
for $t=1, \ldots, \order$.

We denote by $\gamma_{[t]}=\dim \Gr_{h_t}( \ker H_t(\D_{t-1}))$ for $\D_{t-1}\in U_{t-1}$ and $\gamma=\sum_{t=0}^\order \gamma_{[t]}$.
\end{definition}

If the basis $B$ is regular and closed under division, then
$\Hreg_{\hb}$ can be parametrized by rational functions of free
parameters $\barmu$. We present hereafter Algorithm~\ref{alg:RP} to compute
such a parametrization iteratively.
\begin{algorithm}
\caption[RP]{\sc Rational Parametrization - Iteration $t$ }
\label{alg:RP}
\begin{flushleft}\noindent{\bf Input:} $t> 0$,  $B_t=\{\bx^{\bbeta_1}, \ldots, \bx^{\bbeta_{r_t}}\}\subset \CC[\bx]_t$ closed under division and regular,  $\barmu_{t-1}\subset \mu_{t-1}$  and $\Phi_{t-1}: \barmu_{t-1}\mapsto \left({q_{\beta_j, \alpha}(\barmu_{t-1})}\right)_{|\beta_j|\leq t-1, |\alpha|<\mult}$ with $q_{\beta_j, \alpha}\in \QQ(\barmu_{t-1})$ parametrizing a dense subset of $\Hreg_{\hb_{t-1}}$.\\
{\bf Output:} $\barmu_{t}\subset \mu_{t}$ and $\Phi_{t}:\barmu_{t}\mapsto \left({q_{\beta_j, \alpha}}\right)_{|\beta_j|\leq t, |\alpha|<\mult}$, $q_{\beta_j, \alpha}{\in} \QQ(\barmu_{t})$ extending $\Phi_{t-1}$ and parametrizing a dense subset of $\Hreg_{\hb_t}$.\\
\hrule\vspace{2mm}
\textbf{(1)} Let $H_t$ be as in Notation \ref{not:HtJt}, $\nu=[\nu_j^k: j=1, \ldots, r_{t-1}, k=1, \ldots, n]^T$. Decompose $H_t(\Phi_{t-1}(\barmu_{t-1}))\cdot \nu=0$ as
\begin{equation}\label{eq:matrixABC2}
\left[
\begin{array}{c|c|c}
A(\barmu_{t-1})& B(\barmu_{t-1}) & C(\barmu_{t-1})
\end{array}
\right]
\left[
\begin{array}{c}
  \nu'\\
  \nu''\\
  \barnu
\end{array}
\right] = 0,
\end{equation}
where $\nu'$ is associated
to a maximal set of independent columns of
$H_{t}(\Phi_{t-1}(\barmu_{t-1}))$,
$\nu''=\{\nu_j^k: \;\xb^{\beta_{j}+\e_{k}}\in B_{[t]}\}$ and
$\barnu$ refers to the rest of the columns.
 If no such decomposition exists, return ``$B_t$ is not regular''.
\\

\textbf{(2)} For $\nu_j^k\in \nu'$ express $\nu_j^k=\varphi_j^k(  \barnu,\nu'')\in \QQ(\barmu_{t-1})[ \barnu,\nu'']_1$ as the generic solution of the system     $H_t(\Phi_{t-1}(\barmu_{t-1}))\cdot \nu=0$.\\

\textbf{(3)} For $i=r_{t-1}+1,\ldots, r_t$ do:
\begin{itemize}
 \item[(3.1)] Define $\barmu_{[t],i}:= \left\{ \mu_{\beta_i, \beta_j+\e_k}\;:\; \nu_{j,k}\in \barnu\right\}$, $\mu_{[t],i}'=\{\mu_{\beta_{i},\beta_{j}+\e_{k}}: \nu_j^k\in \nu'\}$, $\mu''_{[t],i}=\{\mu_{\beta_{i},\beta_{j}+\e_{k}}: \nu_j^k\in \nu''\}$, and
 $\displaystyle \barmu_t:=\barmu_{t-1}\cup\bigcup\nolimits_{i=r_{t-1}+1}^{r_t} \barmu_{[t],i} .
$
\item[(3.2)]  For  $\mu_{\beta_{i},\beta_{j}+\e_{k}}\in \mu''_{[t],i}$
set $q_{\beta_i, \beta_j+\e_k}=\mu_{\beta_{i},\beta_{j}+\e_{k}}= 1$ if
$\beta_{i}=\beta_{j}+\e_{k}$ and $0$ otherwise.
\item[(3.3)] For $\mu_{\beta_{i},\beta_{j}+\e_{k}}\in \mu'_{[t],i}$ define
$$q_{\beta_i, \beta_j+\e_k}:= \varphi_j^k(  \barmu_{[t],i},\mu''_{[t],i})\in \QQ (\barmu_t)
$$
\item[(3.4)] For $|\alpha|<r$ and $\mu_{\beta_{i},\alpha}\not\in  \mu_t$ find
  $q_{\beta_i,\alpha}$ using Lemma \ref{lem:coord}.
\end{itemize}
\end{flushleft}
\end{algorithm}

\begin{proposition}\label{prop:param} Let $B=\{\bx^{\beta_1}, \ldots,
  \bx^{\beta_\mult}\}\subset \Rg_{\mult-1}$ be closed under division and
  assume that $B$ is a regular basis. 
  There exist a subset $\barmu\subset \mu$ with $|\barmu|=\gamma$ and rational functions $ q_{\beta_j,\alpha}(\barmu)\in \QQ(\barmu)$ for
$j=1, \ldots, \mult$ and $|\alpha|< \mult$, such that the map $\Phi:\CC^\gamma\rightarrow \Hilb_B$ defined by
$$
\Phi:\barmu \mapsto \left({q_{\beta_j, \alpha}(\barmu)}\right)_{j=1, \ldots, \mult, |\alpha|<\mult}
$$
parametrizes a dense subset of $\Hreg_{\hb}$.
\end{proposition}
\begin{proof}
Let us define, by induction on $t$, parameters $\barmu_{t}$ with $|\barmu_{t}|=\sum_{i=1}^t \gamma_{[i]}$, and a
rational parametrization of a basis
$\Lambda_{1}(\barmu_{t}), \ldots, \Lambda_{r_{t}}(\barmu_{t})$
of a generic element of $ \Hreg_{B_t}$.
For $t=0$, we define $\Lambda_{1}=1$ and
$\barmu_{0}=\emptyset$. Assume that there exist $\barmu_{t-1}\subset \mu_{t-1}$ and
a rational parametrization
$\Lambda_{1}(\barmu_{t-1}), \ldots, \Lambda_{r_{t-1}}(\barmu_{t-1})$  of a
basis dual to  $B_{t-1}$ for a generic element  $\Hilb_{B_{t-1}}$ defined by the map
$$
\Phi_{t-1}: \barmu_{t-1}\mapsto \left({q_{\beta_j, \alpha}(\barmu_{t-1})}\right)_{|\beta_j|\leq t-1, |\alpha|<\mult}.
$$
This means that $\overline{\textup{im}\, \Phi_{t-1}}=\Hilb_{B_{t-1}}$.
Denote by\\ $\D_{t-1}(\barmu_{t-1})\in \Gr_{r_{t-1}} (\QQ(\barmu_{t-1})[\bd]_{t-1})
$ the space spanned by\\
$\{\Lambda_{1}(\barmu_{t-1}), \ldots,
\Lambda_{r_{t-1}}(\barmu_{t-1})\}$ over the fraction field $\QQ(\barmu_{t-1})$.

By Theorem \ref{thm:hilbint} and Lemma \ref{lem:primaldual}, to
define $\barmu_{t}$ and to extend $\D_{t-1}(\barmu_{t-1})$ to $\D_t(\barmu_{t})$, we need to find $\Lambda_{r_{t-1}+1}, \ldots, \Lambda_{r_{t}}$
 of the form
$$
\Lambda_{i} = \sum_{j=1}^{r_{t-1}}\sum_{k=1}^{n}\mu_{\beta_i,\beta_j+\e_k} \tint{k} \Lambda_{j}(\barmu_{t-1})\;\; i= r_{t-1}+1, \ldots, r_t,
$$
 satisfying the system of equations \eqref{eq:(1)} and
\eqref{eq:(2)}, i.e. such that
$$\Lambda_i\in \ker H_t(\barmu_{t-1})
\textup{  for } i= r_{t-1}+1, \ldots, r_t,
$$
where $H_t(\barmu_{t-1})=H_t\left(\Phi_{t-1}(\barmu_{t-1})\right)$ and
Equations \eqref{eq:fperp} are satisfied.
%
Since $B$ is a regular basis, the kernel of
$H_{t}(\barmu_{t-1})$ over $\QQ(\barmu_{t-1})$ contains a subspace
$\D_{[t]}$ of dimension $h_{t}= |B_{[t]}|$ with
$\Delta_{B_{[t]}}(\D_{[t]})\neq 0$. Therefore, the systems
$H_{t}(\barmu_{t-1})\, \nu=0$ with $\nu=[\nu_j^k: j=1, \ldots, r_{t-1}, k=1, \ldots, n]^T$ can
be decomposed as
\begin{equation}\label{eq:matrixABC}
\left[
\begin{array}{c|c|c}
A(\barmu_{t-1})& B(\barmu_{t-1}) & C(\barmu_{t-1})
\end{array}
\right]
\left[
\begin{array}{c}
  \nu'\\
  \nu''\\
  \barnu
\end{array}
\right] = 0,
\end{equation}
where $\nu'$ is associated
to a maximal set of independent columns of $H_{t}(\barmu_{t-1})$, $\nu''=\{\nu_j^k: \;\xb^{\beta_{j}+\e_{k}}\in B_{[t]}\}$  and
$\barnu$ is associated to the remaining set of columns. Note that $|\barnu|=\dim(\ker H_{t}(\barmu_{t-1}))-h_t$.
Thus, $\nu'' \cup \barnu$ is the set of free variables of the homogeneous system $H_{t}(\barmu_{t-1})\, \nu=0$ and  a general solution is such that
the variables in $\nu'$ are linear functions of the variables in
$\nu''$ and $\barnu$, with rational coefficients in $\barmu_{t-1}$.

We obtain the coefficients of $\Lambda_{r_{t-1}+1}, \ldots, \Lambda_{r_t}$ that satisfy equations \eqref{eq:(1)} and
\eqref{eq:(2)} and \eqref{eq:fperp} from the general solutions of $H_{t}(\barmu_{t-1})\, \nu=0$ by further specializing the variables in $\nu''$
to $0$'s and $1$'s, according the duality conditions.
Define
$$\barmu_{[t],i}:= \left\{ \mu_{\beta_i, \beta_j+\e_k}\;:\; \nu_{j,k}\in \barnu\right\}\subset \mu_{[t]} \, . $$
Thus, the parameters in $\mu_{[t]}$  are  linear functions of $\barmu_{[t],i}$ with
rational coefficients in $\barmu_{t-1}$. The denominator in these
coefficients is a factor of the numerator of a maximal non-zero minor of
$A(\barmu_{t-1})$. Note that the rest of the coefficients of $\Lambda_i$  are polynomial functions of the parameters $\mu_{t-1}\cup \mu_{[t]}$ by Lemma \ref{lem:coord}. Define
 $$\barmu_t:=\barmu_{t-1}\cup\bigcup_{i=r_{t-1}+1}^{r_t} \barmu_{[t],i} .
$$
Thus, we get a parametrization of the coefficients of $\Lambda_{r_{t-1}+1}(\barmu_t), \ldots, \Lambda_{r_t}(\barmu_t)$ in terms of $\barmu_t$,  which defines the degree $t$ part of the map $\Phi_t:\barmu_t\mapsto (q_{\beta_j,\alpha}(\barmu_t))_{|\beta_j|\leq t, |\alpha|<r}$.   For $\D_{t}\in \Hilb_{B_{t}}$, the coefficients of its basis dual to $B_t$  can be parametrized by $\Phi_{t}$  for parameter values $\barmu_{t}$ such that a maximal
non-zero minor of $A(\barmu_{t-1})$ in $\QQ(\barmu_{t-1})$ does not vanish.

Note that the number of new parameters introduced is
$$
|\barmu_t\setminus \barmu_{t-1}|=(r_{t}-r_{t-1})\cdot|\barmu_{[t],i}|=h_t\left(\dim \ker H_t(\barmu_{t-1})-h_t\right)
$$
which is equal to $\gamma_{[t]}=\dim \Gr_{h_t}( \ker H_t(\barmu_{t-1}))
=\dim \Gr_{h_t}( \ker H_t(\D_{t-1}))$ for $\D_{t-1}$ generic
in $U_{t-1}$ as claimed.

To prove that $\Phi_{t}$ parametrizes a dense subset of the projective
variety $\Hreg_{\hb_{\vspanh{t}}}$, note that  the image ${\rm
  im}(\Phi_{t})$ of $\Phi_{t}$ is a subset of
$\Hilb_{\hb_{\vspanh{t}}}$, the Zariski closure $V_{t}$ of  ${\rm
  im}(\Phi_{t})$ is an irreducible subvariety of
$\Hilb_{\hb_{\vspanh{t}}}$. Furthermore, its projection
$\pi_{t-1}(V_{t}) \subset \Hilb_{\hb_{\vspanh{t-1}}}$ is the closure of
the image of ${\rm im}(\Phi_{t-1})$ since if $\D_{t}= {\rm im} \Phi_{t}(\barmu_{t}^{*})$
then $\D_{t-1} = \D_{t} \cap \DRg_{t-1}= \Phi_{t-1}(\barmu_{t-1}^{*})$.
By induction hypothesis,
$$\pi_{t-1}(V_{t})= \overline{{\rm im}
    \Phi_{t-1}}= \Hreg_{\hb_{\vspanh{t-1}}} \, . $$
Thus, $V_{t}$ is the irreducible
component of $\Hilb_{\hb_{\vspanh{t}}}$ which projects onto
$\Hreg_{\hb_{\vspanh{t-1}}}$, that is $\Hreg_{\hb_{\vspanh{t}}}$.
\end{proof}



\begin{definition}\label{def:minor-regular}
We denote by $\Areg_{t}(\mu)$ {\em a} maximal square submatrix of $A$ in \eqref{eq:matrixABC2} such that
$\det(\Areg_{t}(\barmu_{t-1})) \neq 0$.
\end{definition}
The size of $\Areg_{t}(\mu)$ is the size of $\nu'$ in
\eqref{eq:matrixABC2}, that is the maximal number of independent
columns in $H_{t}(\barmu_{t-1})$.
Given an element $\D=\Lambda_{1}\wedge \cdots \wedge \Lambda_{\mult}\in
\Gr_{\mult}(\DRg_{r-1})$, in order to check that $\D$ is regular for
$B$, it is sufficient to check first that $\Delta_{B}(\D)\neq 0$ and
secondly that $|\Areg_{t}(\mu)|\neq 0$ for all $t=0, \ldots, \order$,
where $\mu =(\mu_{\beta,\alpha})$ is the ratio of Pl\"ucker
coordinates of $\D$ defined by the formula \eqref{eq:plucker}.

\section{Newton's iterations}\label{sec:Newton}

In this section we describe the extraction of a square, deflated
system that allows for a Newton's method with quadratic convergence.
We assume that the sole input is the equations $\fb=(f_1, \ldots, f_N)\in\CC[\bx]^N$,
an approximate point $\xi\in\CC^n$ and a tolerance $\varepsilon>0$.

Using this input we first compute an approximate primal-dual pair
$(B,\,  \bLambda)$ by applying the iterative
Algorithm~\ref{AlINT}. The rank and kernel vectors of the matrices
$K_t$ (see Algorithm~\ref{AlINT}) are computed numerically within
tolerance $\varepsilon$, using SVD.
Note that here and in Section 6 we do not need to certify the SVD computation but we are only using SVD to certify that some matrices are full rank by checking that the distance to the variety of singular matrices is bigger than the perturbation of the matrix. Thus we need a weaker test, which relies only on a lower bound of the smallest singular value.

The algorithm returns a basis
$B=\{\bx_\xi^{\beta_1}, \ldots, \bx_\xi^{\beta_\mult}\}$ with exponent vectors $E=\{\beta_1, \ldots, \beta_\mult\}$,
as well as approximate values for the parameters $\mu=\{
\mu_{\bbeta_i, \bbeta_j+\e_k}\;:\;|\beta_j|<|\beta_i|\in E, \;k=1, \ldots, n\}$.
 These parameters will be used as a
starting point for Newton's iteration. Note that, by looking at
$B$, we can also deduce the multiplicity $\mult$, the maximal order
$\order$ of dual differentials, the sequences  $r_t=|B_t|$,  and  $h_t=|B_{[t]}|$ for $t=0, \ldots, \order$.

Let $F$ be the deflated system with variables $(\bx,\mu)$ defined by
the relations \eqref{eq:newLambda} and Equations \eqref{eq:commnew},
\eqref{eq:orth} and  \eqref{eq:fperp} i.e.
{\footnotesize
$$
  {F(\bx{,}\mu){=}}
  \left\{\hspace{-.19cm}
\begin{array}{ll}
\sum\limits_{|\beta_s|<|\beta_j|<|\beta_i|}\mu_{\bbeta_i,
  \bbeta_j+\e_k}\mu_{\beta_j, \beta_s+\e_l}-\mu_{\bbeta_i,
  \bbeta_j+\e_l}\mu_{\beta_j, \beta_s+\e_k} {=}0 &{(a)}\label{eq:(a)}\\
\hspace{0.5cm}  \text{ for all } i=1, \ldots, \mult, \,
  |\beta_s|<|\beta_i|, k\neq l\in \{1, \ldots, n\}\nonumber & \\
\mu_{\bbeta_i, \bbeta_j+\e_k}=\begin{cases} 1 & \text{ for } \bbeta_i= \bbeta_j+\e_k\\
0 & \text{ for } \bbeta_j+\e_k \in E, \; \bbeta_i \neq \bbeta_j+\e_k
,\end{cases}& (b)\label{eq:(b)}\\
\Lambda_i ( {f}_{j})=0,\quad i=1, \ldots, \mult,\, j=1, \ldots, N. & (c)\label{eq:(c)}
  \end{array}
\right.
$$
}{}
Here $\Lambda_1=1_\bx$ and $\Lambda_i=\sum_{|\beta_j|<|\beta_i|} \sum_{k=1}^n \mu_{\bbeta_i, \bbeta_j+\e_k}\tint{k}{\Lambda_j}\in\CC[\mu][\bd_\bx]$ denote dual elements with
parametric coefficients defined recursively.  Also, if $\Lambda_i=\sum_{|\alpha|\leq |\beta_i|} \mu_{\beta_i, \alpha}\frac{\bd_\bx^\alpha}{\alpha!}$ then
$$
\Lambda_i ( {f}_{j}) =\sum_{|\alpha|\leq |\beta_i|} \mu_{\beta_i, \alpha}\frac{\bpartial^\alpha(f_j)(\bx)}{\alpha!}
$$
which is in $ \CC[\bx, \mu]$ by Lemma \ref{lem:coord}.  Note, however, that
\eqreftag{(a)} and
\eqreftag{(b)} are polynomials in $\CC[\mu]$, only \eqreftag{(c)}
depends on $\bx$ and $\mu$.  Equations
\eqreftag{(b)} define a simple substitution
into some of the parameters $\mu$.  Hereafter, we explicitly substitute them and
eliminate this part \eqreftag{(b)} from the equations we consider and reducing the
parameter vector $\mu$.

By Theorem \ref{thm:isolated}, if $B$ is a graded primal basis for $\fb$ at the root $\xi^*$ then the above overdetermined system
has a simple root at a point $(\xi^*,\mu^*)$.

To extract a square subsystem defining the simple root $(\xi^*,\mu^*)$
in order to certify the convergence, we choose a maximal set of
equations whose corresponding rows in the Jacobian are linearly
independent. This is done by extracting first a maximal set of
equations in \eqreftag{(a)} with linearly independent rows in the
Jacobian. For that purpose, we use the rows associated to the
maximal invertible matrix $\Areg_{t}$ (Definition \ref{def:minor-regular}) for each new
basis element $\Lambda_{i} \in \D_{[t]}$ and $t=1,\ldots,\mult$. We
denote by $G_{0}$ the subsystem of \eqreftag{(a)} that correspond to rows of $\Areg_{t}$.

We complete the system of independent equations $G_{0}$ with equations from
\eqreftag{(c)}, using a QR decomposition and thresholding on
the transposed Jacobian matrix of $G_{0}$ and \eqreftag{(c)} at the approximate
root. Let us denote by $F_{0}$ the resulting square system, whose
Jacobian, denoted by $J_0$, is invertible.

For the remaining equations $F_{1}$ of \eqreftag{(c)}, not used to construct
the square system $F_{0}$, define $\Omega=\{(i,j): \Lambda_{i}( f_{j})\in F_1\}$.  We introduce new parameters $\epsilon_{i,j}$ for $(i,j)\in \Omega$ and we
consider the perturbed system
$$
{f}_{i, \epsilon}=  {f}_{i} - \sum_{j\mid (i,j)\in \Omega}\epsilon_{i,j}\,\xb_{\xi}^{\beta_{j}}.
$$
The perturbed system is $\fb_{\epsilon} = \fb - \ef\, B$, where
$\ef$ is the $N\times r$ matrix with $[\ef]_{i,j}=\epsilon_{i,j}$ if
$(i,j)\in \Omega$ and $[\ef]_{i,j}=0$ otherwise.
Denote by $F(\xb, \mu, \epsilon)$ obtained from $F(\bx, \mu)$ by replacing 
$\Lambda_j ( {f}_{i})$  by $\Lambda_j ( {f}_{i,
  \epsilon})$  for $j=1, \ldots, \mult, i=1, \ldots, N$.
Then  the equations used to construct the square Jacobian $J_{0}$
are unchanged. The remaining equations are of the form
$$
\Lambda_j ( {f}_{i,
  \epsilon}) =\Lambda_{j}({f}_{i}) - \epsilon_{i,j}=0 \quad (i,j)\in \Omega.
$$Therefore the Jacobian of the complete system $F(\xb, \mu, \epsilon)$
is a square
invertible matrix of the form
$$
J_{\epsilon}:=\left(
\begin{array}{cc}
  J_{0} & 0\\
  J_{1} &  \mathrm{Id}\\
\end{array}
\right)
$$
where $J_{1}$ is the Jacobian of the system $F_{1}$ of polynomials
$\Lambda_{j}({f}_{i})\in \CC[\xb,\mu]$ with $(i,j)\in \Omega$.

Since $J_{\epsilon}$ is invertible, the square extended system $F(\xb, \mu, \epsilon)$
has an isolated root $(\xi^{*}, \mu^{*}, \epsilon^{*})$ corresponding to the
isolated root $(\xi^{*}, \mu^{*})$ of  the square system $F_{0}$. Furthermore,
 $\Lambda_{j}^{*}(f_{i})=\epsilon^*_{i,j}=0$ for
$(i,j)\in \Omega$. Here $\Lambda^*_1, \ldots, \Lambda^*_\mult\in \CC[\bd_{\xi^*}]$ are defined from  $(\xi^*,\mu^{*})$ recursively by
\begin{eqnarray}\label{eq:recursive}
\Lambda_1^*=1_{\xi^*}\text{ and }\Lambda_i^*=\sum_{|\beta_j|<|\beta_i|} \sum_{k=1}^n \mu^*_{\bbeta_i, \bbeta_j+\e_k}\tint{k}{\Lambda^*_j}.
\end{eqnarray}

We have the following property:
\begin{theorem}\label{thm:pert}\label{thm:perturbed system}
If the Newton iteration {\small $$(\xi_{k+1}, \mu_{k+1}) = (\xi_{k}, \mu_{k})-
J_{0}(\xi_{k}, \mu_{k})^{-1} F_{0}(\xi_{k}, \mu_{k}),$$} starting from a point $(\xi_{0}, \mu_{0})$  converges
when $k\rightarrow \infty$, to a point $(\xi^{*},\mu^{*})$ such that
$B$ is a regular basis for the inverse system $\D^{*}$ associated to
$(\xi^{*},\mu^{*})$ and $\D^{*}$ is complete for $\fb$,
then there exists a perturbed system ${f}_{i, \epsilon^{*}}= {f}_{i} - \sum_{j\mid (i,j)\in
  \Omega}\epsilon^{*}_{i,j}\,\xb_{\xi^{*}}^{\beta_{j}}$ with
$\epsilon^{*}_{i,j} = \Lambda_{j}^{*}(f_{i})$ such that $\xi^{*}$ is a
multiple root of $ f_{i,\epsilon^*}$ with the multiplicity structure
defined by $\mu^{*}$.
\end{theorem}
\begin{proof}
If the sequence  $(\xi_{k}, \mu_{k})$ converges to the fixed point $(\xi^{*},
\mu^{*})$, then we have $F_{0}(\xi^{*},\mu^{*})=0$ and in particular,
$G_{0}(\xi^{*},\mu^{*})=0$ where $G_{0}(\xi^{*},\mu^{*})=0$ is the
subset of equations selected from \eqreftag{(a)}.

As $\mu^{*}$ is regular for $B$, if it satisfies
$G_{0}(\xi^{*},\mu^{*})=0$, it must satisfy all equations
\eqreftag{(a)}. Therefore $\mu^{*}$ defines a point
$\D^{*}=\Lambda^{*}_{1}\wedge \cdots \wedge \Lambda^{*}_{\mult}\in
\Hilb^{reg}_{B}$.

As $(\Lambda^{*}_{i})$ is a basis of $\D^{*}$ dual to $B$
and ${f}_{i, \epsilon^{*}}= {f}_{i} - \sum_{j: (i,j)\in
  \Omega}\epsilon^{*}_{i,j}\,\xb_{\xi^{*}}^{\beta_{j}}$ with
$\epsilon^{*}_{i,j} = \Lambda_{j}^{*}(f_{i})$ for $(i,j)\in \Omega$,
we have that if $(i,j)\in \Omega$ then
$\Lambda^*_{j}(f_{i,\epsilon^{*}})=\Lambda^{*}_{j}({f}_{i})-\epsilon^{*}_{i,j}=0$.
Otherwise $\Lambda^*_{j}(f_{i,\epsilon^{*}})= \Lambda^{*}_{j}({f}_{i})$,
since it is one of the
equations selected in \eqreftag{(c)} to construct the system $F_{0}$
and $F_{0}(\xi^{*},\mu^{*})=0$.
This shows that
$$\fb_{\epsilon^{*}}=(f_{i,\epsilon^{*}})_{i=1}^N\subset (\D^{*})^{\perp} \,.$$
Since $\fb_{\epsilon^{*}}$ is obtained from $\fb$ by adding elements
in $B$, the system \eqreftag{(c)}, at order $\order+1$
for $\fb_{\epsilon^{*}}$ and $\fb$ are equivalent.
Thus $\D^{*}$ is complete for $\fb$ and $\fb_{\epsilon}$ and
$\D^{*}= (\fb_{\epsilon^{*}}) ^{\perp} \cap \CC[\bd_{\xi^*}]$ is the inverse
system at $\xi^{*}$ of the system $\fb_{\epsilon^{*}}$.
\end{proof}
\section{Certification}

In this section we describe  how to certify that the Newton iteration
defined in Section \ref{sec:Newton} quadratically converges to a point
that defines an exact root with an exact multiplicity structure of a
perturbation of the input polynomial system $\fb$. More precisely, we
are given
$\fb=(f_1, \ldots, f_N)\in \CC[\bx]^N$, $B=\{\bx^{\beta_1}, \ldots,
\bx^{\beta_\mult}\}\subset \CC[\bx]$ in increasing order of degrees
and closed under division,  $\order:=|\beta_\mult|$. We are also given
the deflated  systems $F(\bx,\mu)$, its square subsystem
$F_0(\bx, \mu)$ defined in Section \ref{sec:Newton} and $F_1(\bx, \mu)$ the remaining equations in $F(\bx,\mu)$. Finally, we are
given $\xi_0\in \CC^n$ and $\mu_0= \{\mu^{(0)}_{\beta_i, \beta_j+\e_k}\in
\CC\;:\; i, j=1, \ldots, \mult, |\beta_j|<|\beta_i|, k=1, \ldots,
n\}$. Our certification will consist of a symbolic and a numeric part:

\vspace{2mm}\noindent{}\textbf{Regularity certification.}
We certify that $B$ is regular (see Definition \ref{def:non-deg}). This part of the certification is purely symbolic and inductive on $t$. Suppose for some  $t-1< \order$ we certified that $B_{t-1}$ is regular and computed the parameters $\barmu_{t-1}$ and the parametrization
$$
\Phi_{t-1}: \barmu_{t-1}\mapsto \left(q_{\beta_i, \alpha}(\barmu_{t-1})\right)_{|\beta_i|\leq t-1, |\alpha|\leq t-1}
$$
(Algorithm \ref{alg:RP}). Then to prove that $B_t$ is regular, we consider the coefficient matrix $H_t$ of equations \eqref{eq:(1)} and \eqref{eq:(2)}. We substitute the parametrization $\Phi_{t-1}$ to get the matrices $H_t(\barmu_{t-1})$.
We symbolically prove that the rows of $\Areg_t(\barmu_{t-1})$ (Definition \ref{def:minor-regular}) are linearly independent and span all rows of $H_t(\barmu_{t-1})$ over $\QQ(\barmu_{t-1})$. If that is certified, we compute the parameters $\barmu_t$ and the parametrization $\Phi_t:\barmu_{t}\mapsto \left(q_{\beta_i, \alpha}(\barmu_{t})\right)_{|\beta_i|\leq t, |\alpha|\leq t}$ as in
Algorithm \ref{alg:RP} inverting the square submatrix $\Areg_t$ of $H_t$ such that the denominators of $q_{\beta_i, \alpha}$ for $|\beta_i|=t$ divide $\det(\Areg_t(\barmu_{t-1}))\neq 0$.

\vspace{2mm}\noindent{}\textbf{Singularity certification.}
\begin{enumerate}
\item[(C1)] We certify that the Newton iteration for the square system
  $F_0$ starting from  $(\xi_0, \mu_0)$ quadratically converges to some
  root $(\xi^*, \mu^*)$ of $F_0$, such that $\|(\xi_0, \mu_0)-(\xi^*,
  \mu^*)\|_2\le \atbeta$, using  $\alpha$-theory.

\item[(C2)] We certify that $\D^*={\rm span}(\bLambda^*)$ is regular
  for $B$ (see Definition  \ref{def:non-deg}), by checking
  that $|\Areg_{t}(\mu^{*})|\neq 0$ for $t=1,\ldots, \order$
  (See Definition \ref{def:minor-regular}), using the Singular Value
  Decomposition of $\Areg_{t}(\mu_0)$ and the distance bound $\atbeta$
  between $\mu^{*}$ and $\mu_0$.

\item[(C3)] We certify that $\bLambda^*$ is complete for $\fb$ at
  $\xi^*$ (see Definition \ref{def:compl}), where $\bLambda^*\subset
  \CC[\bd_{\xi^*}] $ is the dual systems defined from  $(\xi^*,
  \mu^*)$ recursively as in \eqref{eq:recursive}. This is done  by checking that
  $\ker \Acomp(\xi^{*},\mu^{*}) =\{ 0\}$
  (See Definition \ref{def:compl}), using the Singular Value
  Decomposition of $\Acomp(\xi_0,\mu_0)$ and the distance bound $\atbeta$
  between $(\xi^{*},\mu^{*})$ and $(\xi_0,\mu_0)$.

\end{enumerate}
Let us now consider for a point-multplicity structure pair $(\stxi,\stmu)$ $\atgamma := \sup_{k\ge 2} \| DF_{0}^{-1}(\stxi,\stmu)
{D^{k}F_{0}(\stxi,\stmu)\over k!}\|^{{1\over k-1}}$, $\;\;\atbeta := 2 \|
DF_{0}^{-1}(\stxi,\stmu)  \,F_{0}(\stxi,\stmu) \|$, $\atalpha:=\atbeta\, \atgamma$
and for a matrix function $A(\xi,\mu)$, let $\L_{1}(A; \stxi,\stmu;b)$ be a bound on
its Lipschitz constant in the ball $\B_{b}(\stxi,\stmu)$ of radius $b$ around
$(\stxi,\stmu)$ such that
$\|A(\xi,\mu)-A(\stxi,\stmu)\|\leq \L_{1}(A;\stxi,\stmu;b)\, \|(\xi,\mu)-(\stxi,\stmu)\|$
for $(\xi,\mu)\in \B_{b}(\stxi,\stmu)$. For a matrix $M$, let
$\sigma_{min}(M)$ be its smallest singular value. We have the following result:

\begin{theorem}\label{thm:final}\label{th:pert}
Let $B=\{\bx^{\beta_1}, \ldots,
\bx^{\beta_\mult}\}\subset \CC[\bx]$ be closed under division and suppose $B$ is regular. Suppose that $\atalpha < \atalpha_{0} := 0.26141$,\\
$\L_{1}(\Acomp;\stxi,\stmu; \atbeta) \,\atbeta <$ $ \sigma_{\min}(\Acomp(\stxi,\stmu))$
and for $t=1, \ldots, \order$ it holds that
$\L_{1}(\Areg_{t};\stmu; \atbeta)\, \atbeta< \sigma_{min}(\Areg_{t}(\stmu))$. Then the Newton iteration on the square system $F_{0}$ starting from
$(\stxi,\stmu)$ converges quadratically to a point $(\xi^{*},\mu^{*})$ corresponding
to a multiple point $\xi^{*}$ with multiplicity structure
$\mu^{*}$ of the perturbed system $\fb_{\epsilon^{*}}=\fb-\ef^{*}
B_{\xi^{*}}$ such that
$\|\ef^{*}\| \leq \|F_{1}(\stxi,\stmu)\|+
\L_{1}(F_{1};\stxi,\stmu;\atbeta)\, \atbeta$, where $B_{\xi^*}=\{\bx_{\xi^*}^{\beta_1}, \ldots,
\bx_{\xi^*}^{\beta_\mult}\}$.
\end{theorem}
\begin{proof}
By the $\alpha$-theorem \cite{blum_complexity_1998}[Chap. 8, Thm. 1], the Newton iteration on $F_{0}$
starting from $(\stxi,\stmu)$ converges quadratically to a point $(\xi^{*},\mu^{*})$
such that
$$\|(\xi^{*},\mu^{*})-(\stxi,\stmu)\|< \atbeta \, .$$
We deduce
that
{\small$$\begin{array}{rcl}
\|\Acomp(\xi^{*},\mu^{*})-\Acomp(\stxi,\stmu)\|&\leq&
                                                    \L_{1}(\Acomp;\stxi,\stmu; \atbeta)\, \|(\xi^{*},\mu^{*})-(\stxi,\stmu)\| \\
  &<& \sigma_{min}(\Acomp(\stxi,\stmu)).
              \end{array}
            $$}
Therefore $\Acomp(\xi^{*},\mu^{*})$ is within a ball around
$\Acomp(\stxi,\stmu)$ of matrices of maximal rank, since
$\sigma_{min}(\Acomp(\stxi,\stmu))$ is the distance between
$\Acomp(\stxi,\stmu)$ and the set of matrices not of maximal rank.

\noindent Thus $\ker \Acomp(\xi^{*},\mu^{*})= \{0\}$. A similar argument shows that
$|\Areg_{t}(\mu^{*})|\neq 0$ for $t=1, \ldots, \order$.
By Theorem \ref{thm:pert}, $(\xi^{*},\mu^{*})$ defines a multiple root $\xi^{*}$ with multiplicity structure $\mu^{*}$ for the perturbed system
$\fb_{\epsilon^{*}} =\fb - \ef^{*} B_{\xi^{*}}$
with
{\small\begin{eqnarray*}
\|\ef^{*}\|= \|F_{1}(\xi^{*},\mu^{*})\|\hspace{-.2cm} &\le&
\hspace{-.2cm}\|F_{1}(\stxi,\stmu)\|\ + \|F_{1}(\xi^{*},\mu^{*})-F_{1}(\xi^{*},\mu^{*})\|\\
&\le&
\hspace{-.2cm}\|F_{1}(\stxi,\stmu)\| + \L_{1}(F_{1};\stxi,\stmu;\atbeta)\, \|(\xi^{*},\mu^{*})-(\stxi^{*},\stmu^{*})\|\\
&\le&
\hspace{-.2cm}\|F_{1}(\stxi,\stmu)\| + \L_{1}(F_{1};\stxi,\stmu;\atbeta)\, \atbeta. \qedhere
\end{eqnarray*}}
\end{proof}

\section{Experimentation}

In this section we work out some examples with (approximate)
singularities.  The experiments are carried out using Maple, and our
code is publicly available at
\url{https://github.com/filiatra/polyonimo}.

\begin{example}\label{ex:7.3}
  We consider the equations
\begin{equation*}
  f_1 = x_1^3 + x_2^2 + x_3^2 - 1,\,
      f_2 = x_2^3 + x_1^2 + x_3^2 - 1,\,
      f_3 = x_3^3 + x_1^2 + x_2^2 - 1,
\end{equation*}
the approximate root $\bm\xi_0=(0.002,1.003,0.004)$
and threshold $\varepsilon= 0.01$.
In the following we use 32-digit arithmetic for all computations.

We shall first compute a primal basis using
Algorithm~\ref{AlINT}. In the first iteration we compute
the $3\times 3$  matrix $K_1=K_1(\bm\xi_0)$.
The elements in the kernel of this matrix consists of elements of the form
$\Lambda = \nu_{1}^1 d_1 + \nu_{1}^2 d_2 + \nu_{1}^3 d_3$.
The singular values of $K_1(\bm\xi_0)$ are
$(4.1421, 0.0064, 0.0012)$, which implies a two-dimensional kernel, since
two of them are below threshold $\varepsilon$.
The (normalized) elements in the kernel are
$\tilde \Lambda_2 = d_1 - 0.00117 d_2$ and
$\tilde \Lambda_3 = d_3 - 0.00235 d_2$.
Note that $d_2$ was not chosen as a leading term. This is due to pivoting used
in the numeric process, in order to avoid leading terms with coefficients
below the tolerance $\varepsilon$.
The resulting primal basis $B_{1}=\{1,x_{1},x_{3}\}$ turns out to be closed under derivation.

Similarly, in degree $2$ we compute one element
$\tilde \Lambda_4 =
d_1 d_3 - 0.00002 d_1^2 - 0.00235 d_1 d_2 + 5.5\cdot 10^{-6} d_2^2 - 0.00117\cdot d_2 d_3 - 0.00002 d_3^2 + 5.9\cdot 10^{-6}d_2
$.

In the next step, we have $\ker K_{3}=\{0\}$, since the minimum singular value
is $\sigma_{\min}=0.21549$,
therefore we stop the process, since the computed
dual is approximately complete (cf.~Definition~\ref{def:compl}).  We
derive that the approximate multiple point has multiplicity $r=4$ and
one primal basis is $B=\{1,x_1,x_3,x_1 x_3\}$.

The parametric form of a basis of $\D_{1}$  is
$\ker K_1  = \langle \Lambda_2=d_1 + \mu_{2,1} d_2, \Lambda_3=d_3 + \mu_{3,1} d_2 \rangle$.
Here we incorporated \eqref{eq:orth}, thus fixing some of the
parameters according to primal monomials $x_1$ and $\, x_3$.

The parametric form of the matrix $K_2(\bm\xi,\bm\mu)$
of the integration method at degree 2 is
{\footnotesize$$
\kbordermatrix{ &\nu_{1}^1 &\nu_{1}^2 &\nu_{1}^3 &\nu_{2}^1 &\nu_{2}^2 &\nu_{2}^3 &\nu_{3}^1 &\nu_{3}^2 &\nu_{3}^3 \\
\eqref{eq:commnew}& 0& 0& 0& 0& \textbf{0}& -\mu_{2,1\,}& \textbf{0}& \textbf{1}& -\mu_{3,1\,} \\
\eqref{eq:commnew}& 0& 0& 0& 0& \textbf{0}& -1& \textbf{1}& \textbf{0}& 0 \\
\eqref{eq:commnew} & 0& 0& 0& \mu_{2,1\,}& \textbf{-1}& 0& \bm{\mu_{3,1\,}}& \textbf{0}& 0 \\
\Lambda(f_1)& 3 \xi_1^2& 2 \xi_2& 2 \xi_3& 3 \xi_1& \mu_{2,1\,}& 0& 3 \xi_1& \mu_{3,1\,}& 1\\
\Lambda(f_2)& 2 \xi_1& 3 \xi_2^2& 2 \xi_3& 1& 3 \mu_{2,1\,} \xi_2& 0& 0& 3 \mu_{3,1\,} \xi_2& 1\\
\Lambda(f_3)& 2 \xi_1& 2 \xi_2& 3 \xi_3^2& 1& \mu_{2,1\,}& 0& 0& \mu_{3,1\,}& 3 \xi_3
} \, ,
$$}
where the columns correspond to the parameters in the expansion \eqref{eq:closed}:
{\scriptsize
\begin{eqnarray*}
&\Lambda_{4}  
=
\nu_1^1 d_1
+ \nu_1^2 d_2 + \nu_1^3 d_3
+ \nu_2^1 d_1^2  + \nu_2^2 (d_1d_2+\mu_{2,1\,} d_2^2)\nonumber\\
&
+ \nu_2^3 (d_1d_3 + \mu_{2,1\,} d_3 d_2)
+ \nu_3^1 (\mu_{3,1\,} d_1 d_2) + \nu_3^2 (\mu_{3,1\,} d_2^2) + \nu_3^3 (d_3^2+\mu_{3,1\,} d_2d_3)\nonumber
\end{eqnarray*}}
Setting $\Lambda_4(x_1x_3)=1$ and $\Lambda_4(x_1)=\Lambda_4(x_3)=\Lambda_4(1)=0$,
we obtain $\nu_{1}^{1}=\nu_{1}^{3}=0$ and $\nu_{2}^{3}=1$.
The dual element of order $2$ has the parametric form
\begin{align}
\Lambda_4 &=
d_1 d_3 +
\mu_{4,1\,} d_2 + \mu_{4,2\,} d_1^2 +
\mu_{4,3\,} d_1 d_2 +
\mu_{4,6\,} d_3^2 + \\
&+ (\mu_{2,1\,}  + \mu_{3,1\,}\mu_{4,6\,}) d_2  d_3
+ (\mu_{2,1\,} \mu_{4,4\,} +  \mu_{3,1\,} \mu_{4,5\,}) d_2^2 \nonumber
\end{align}
($\nu_{1}^{2}=\mu_{4,1}, \nu_{2}^{1}=\mu_{4,2},\nu_{2}^{2}=\mu_{4,3}, \nu_{3}^{1}=\mu_{4,4},
\nu_{3}^{2}=\mu_{4,5}, \nu_{3}^{3}=\mu_{4,6}$).
Overall 8 parameters are used in the representation of $\mathcal D_2$.

The highlighted entries of $K_2(\bm\xi,\bm\mu)$ form
the non-singular matrix $\Areg_2$ in Definition~\ref{def:minor-regular}, therefore $\D_2$ is regular for $B$ (cf. Definition~\ref{def:non-deg}).
We obtain the polynomial parameterization
$\mu_{4,3}=\mu_{2,1\,} \mu_{4,2\,} + \mu_{3,1\,}, \mu_{4,4}=1, \mu_{4,5\,}=\mu_{2,1\,}  + \mu_{3,1\,} \mu_{4,6\,}$
with the free parameters $\bar{\mu}=(\mu_{2,1}, \mu_{3,1}, \mu_{4,1},$ $\mu_{4,2}, \mu_{4,6})$. There is no denominator since $\det \Areg_{2}=1$.

We now setup the numerical scheme. The overdetermined and deflated system
$F(\xb, \bm \mu)$ consists of $15$ equations:

\noindent$
\mu_{2,1} \mu_{4,2} + \mu_{3,1} - \mu_{4,3} \ ,
-\mu_{4,4} + 1 \ ,
-\mu_{2,1} \mu_{4,4} - \mu_{3,1} \mu_{4,6} + \mu_{4,5},$ \\
$\Lambda_1(f_1) {=} f_1,  
\Lambda_1(f_2) {=} f_2, 
\Lambda_1(f_3) {=} f_3, 
\Lambda_2(f_1) {=} 2 \mu_{2,1} x_2 + 3 x_1^2 \ ,$  \\
$
\Lambda_2(f_2) {=} 3 \mu_{2,1} x_2^2 + 2 x_1,
\Lambda_2(f_3) {=} 2 \mu_{2,1} x_2 + 2 x_1,
\Lambda_3(f_1) {=} 2 \mu_{3,1} x_2 + 2 x_3, $ \\
$\Lambda_3(f_2) {=} 3 \mu_{3,1} x_2^2 + 2 x_3 \ ,
\Lambda_3(f_3) {=} 2 \mu_{3,1} x_2 + 3 x_3^2 \ , $ \\
$\Lambda_4(f_1) {=} \mu_{2,1} \mu_{4,3} {+} \mu_{3,1} \mu_{4,5} {+} 2 \mu_{4,1} x_2 {+} 3 \mu_{4,2} x_1 {+} \mu_{4,6} \ ,  $\\
$\Lambda_4(f_2) {=} 3 \mu_{2,1} \mu_{4,3} x_2 {+} 3 \mu_{3,1} \mu_{4,5} x_2 {+} 3 \mu_{4,1} x_2^2 {+} \mu_{4,2} {+} \mu_{4,6} \ ,  $\\
$\Lambda_4(f_3) {=} \mu_{2,1} \mu_{4,3} {+} \mu_{3,1} \mu_{4,5} {+} 2 \mu_{4,1} x_2 {+} 3 \mu_{4,6} x_3 {+} \mu_{4,2}$

We now consider $J_F(\bm\xi_0,\bm\mu_0)$. This Jacobian is of full rank, and we can
obtain a maximal minor by
removing $\Lambda_{1}(f_2), \Lambda_{1}(f_3), \Lambda_2(f_3)$ and $\Lambda_3(f_3)$ from $F$. We obtain the square $11\times 11$ system denoted by $F_0$.

The initial point of the Newton iterations is
$\bm \xi_{0}=(0.002,1.003,0.004)$ and the approximation of the variables $\mu_{i,j}$ provided
by the numerical integration method:
$
\bm\mu_0 =
(
-0.00117{,} -0.00235{,} 5.9\cdot 10^{-6}{,} -0.00002{,}
-0.00235{,} 1.0{,} -0.00117{,}
-0.00002) \, .
$

We now use Theorem~\ref{thm:final} to certify the convergence to a
singular system.  We can compute for $(\bm\xi_0,\bm\mu_0)$ the value
$\atbeta\approx 0.01302$.  Moreover, $\sigma_{\min}(\Acomp(\bm\xi_0,\bm\mu_0))=
0.21549$ and the minimum singular value of the highlighted submatrix
of $K_2(\bm\xi_0,\bm\mu_0)$ is equal to one. Therefore $\atbeta$ is at
least one order of magnitude less than both of them, which is
sufficient, since the involved Lipschitz and $\atgamma$ constants are of the order of $1$
for the input polynomials.  In the first iteration we obtain
$\atbeta\approx 0.00011$ which clearly indicates that we are in the
region of convergence.  Indeed, the successive residuals for $4$
iterations are
$
0.00603, 4.0 \cdot 10^{-5} , 2.07\cdot 10^{-9}, 8.6\cdot 10^{-18}, 3.55\cdot 10^{-35} .
$
Clearly, the residual shrinks with a quadratic rate\footnote{The convergence is seen up to machine error. If we increase the accuracy to 150 digits the rate remains quadratic for 7 iterations:
$\dots
3.55\cdot 10^{-35}, 6.78\cdot 10^{-70}, 4.15\cdot  10^{-140}, 5.1\cdot  10^{-281} \, .
$
}.
We obtain $\bm\xi_4= (1.8\cdot 10^{-37} , 1.0, 2.8\cdot  10^{-36})$ and
the overdetermined system is satisfied by this point:
$\| F(\bm\xi_4,\bm\mu_4) \|_{\infty} =  8\cdot 10^{-35}$;
the resulting dual structure is $\D^{*}_2 = \{1, d_1,d_3,d_1 d_3\}$.

\end{example}

\begin{example} 
We demonstrate how our method handles
inaccuracies in the input, and recovers a nearby
system with a true multiple point. Let
  $$
  f_1 =  {x_{{1}}}^{2}+x_{{1}}-x_{{2}}+0.003 \quad,\quad f_2= {x_{{2}}}^{2}+1.004x_{{1}}-x_{{2}} .
  $$
  There is a cluster of three roots around
  $\bm\xi_0=(0.001, -0.002)$. Our goal is to squeeze the cluster down to
  a three-fold real root. We use 32 digits for the computation.
  Starting with $\bm\xi_0$, and a tolerance equal to $10^{-2}$ Algorithm~\ref{AlINT}
  produces an approximate dual
$\bm 1,\, d_1 + 1.00099651 d_2,\,
d_1^2 + 1.00099651 d_1 d_2 + 1.00266222 d_2^2 + 0.99933134 d_2 $ and
identifies the primal basis
$B=\{1,x_1,x_1^2\}$ using pivoting on the integration matrix.
The sole stability condition reads
$\mu_{1,1} - \mu_{2,2} = 0$, and
$\Lambda_1 = \bm 1,\, \Lambda_2 =  d_1 + \mu_{1,1} d_2,\,
\Lambda_3 =  d_1^2 + \mu_{1,1} d_1 d_2 + \mu_{2,1} d_2 + \mu_{2,2}\mu_{1,1} d_2^2$.

  The nearby system that we shall obtain is deduced by the residue in
  Newton's method.  In particular, starting from $\bm\xi_0$, we consider the square system
  given by removing the equations $\Lambda_{1}(f_1)=0$ and $\Lambda_2(f_2)=0$.
  The rank of the corresponding Jacobian matrix remains maximal, therefore such a
  choice is valid. Newton's iterations converge quadratically to the point
  $(\bm\xi_5,\bm\mu_5)=(1.1\cdot 10^{-33},1.2\cdot 10^{-33},1,1,1)$.  The full residual is now
  $$
  F(\bm\xi_5,\bm\mu_5) = (0, 0.003, -10^{-32}, 10^{-32}, 0.004, 0, 0 )\, .
  $$
  This yields a perturbation $\tilde f_1 \approx f_1 - 0.003$ and $\tilde
  f_2 \approx f_2 - 0.004 (x_1-\xi^{*}_{1})$ to obtain a system with an exact multiple
  root at the origin (cf. Th.~\ref{th:pert}).
  Of course, this choice of the square sub-system is not
  unique. By selecting to remove equations $\Lambda_{1}(f_1)=0$ and $\Lambda_{1}(f_2)=0$ instead,
  we obtain $(\bm\xi_5,\bm\mu_5)=(0.00066578, -0.00133245, 1.001, 1.0,
  1.001)$ and the residual $F(\bm\xi_5,\bm\mu_5) = (0, 0.005, 0.002, 0, 0,
  0, 0)$, so that the nearby system
  $$
  f^{*}_1 \approx  {x_{{1}}}^{2}+x_{{1}}-x_{{2}}+0.008,\quad f^{*}_2\approx {x_{{2}}}^{2}+1.004x_{{1}}-x_{{2}} + 0.002 \,
  $$
  has a singularity at the limit point $\bm\xi^{*}\approx
  (0.00066578, -0.00133245)$ described locally by the
  coefficients $\bm\mu^{*}\approx (1.001,1.0,1.001)$.

   Finally, consider the two square sub-systems as above, after changing
   $f_1,f_2$ to define an \emph{exact three-fold root} at the origin
   (i.e. $f_1={x_{{1}}}^{2}+x_{{1}}-x_{{2}},\,
   f_2={x_{{2}}}^{2}+x_{{1}}-x_{{2}}$). Newton's
   iteration with initial point $\bm\xi_0$ on either deflated system converges quadratically to
   $(\bm\xi,\bm\mu)=(\bm 0,\bm 1$).
   This is a general property of the method: exact
   multiple roots and their structure are recovered by this process if $\bm\xi_0$ is
   a sufficiently good initial approximation (cf. Section~\ref{sec:Newton}). We plan to develop this aspect
   further in the future.
\end{example}

\begin{example}
We show some execution details on a set of benchmark examples in taken
from \cite{zeng05}, see also~\cite{mm2014}.
For this benchmark, we are given systems and points with
multiplicities.  We perturb the given points with a numerical perturbation of
order $10^{-2}$.
We use double precision arithmetic and setup Newton's
iteration; with less than $10$ iterations, the root was approximated
within the chosen accuracy.

In Table~\ref{tab:bench}, ``IM'' is the maximal size of the (numeric)
integration matrix that is computed to obtain the multiplicity,
``$\#\mu$'' is the number of new parameters that are needed for
certified deflation, ``SC'' is the number of stability constraints
that were computed  and ``OS'' stands for the size of the
overdetermined system (equations $\times$ variables).  This is the
size of the Jacobian matrix that must be computed and inverted in each
Newton's iteration.
We can observe that the number of parameters required can grow
significantly. Moreover, these parameters induce non-trivial denominators
in the rational functions $q_{\beta_j,\alpha}(\barmu)$ of Prop.~\ref{prop:param}.
for the instances cmbs1, cmbs2 and KSS.


\def\arraystretch{.88}
\begin{table}[ht]
\begin{center}
  \begin{tabular}{|l|c|c|c|c|c|c|c|c|}\hline
System  &$r /  n$ & IM & SC& $\#\mu$  &  OS   \\ \hline 
cmbs1   & 11/3 &$27\times 23$ &75 & 74&  $108\times 77 $  \\ \hline 
cmbs2   & \ 8/3  &$21 \times 17$& 21 &33  & $45\times 36$    \\ \hline
mth191  &\ 4/3  &$10 \times 9$ & 3 & 9 &  $15\times 12$   \\ \hline
decker2 &\ 4/2  &$5 \times  5$ &4 &8 &  $12\times 10$       \\ \hline
Ojika2  &\ 2/3  &$6 \times  5$ &0 &2 & $6\times 5$          \\ \hline
Ojika3  &\ 4/3  &$12 \times 9 $&15 &14 &$27\times 17$      \\ \hline
KSS     & 16/5 &$155\times 65$&510 &362 & $590 \times 367$ \\ \hline
Capr.&\ 4/4  &$22 \times 13$& 6  &15 &  $ 22\times 19$      \\ \hline
Cyclic-9&\ 4/9  &$104 \times 33 $ &36  &40 &  $ 72\times 49$   \\ \hline
\end{tabular}
\end{center}
\caption{Size of required matrices and parameters for deflation.
}
\label{tab:bench}
\end{table}

\end{example}

\subsubsection*{Acknowledgments}
This research was partly supported by the H2020-MSCA-ITN projects
GRAPES (GA 860843) and POEMA (GA 813211) and the NSF grant
CCF-1813340.

\bibliographystyle{acm}

\begin{thebibliography}{}

\end{thebibliography}


\begin{thebibliography}{10}

\bibitem{AyyildizAkogluCertifyingsolutionsoverdetermined2018}
{\sc Ayyildiz~Akoglu, T., Hauenstein, J.~D., and Szanto, A.}
\newblock Certifying solutions to overdetermined and singular polynomial
  systems over {{Q}}.
\newblock {\em Journal of Symbolic Computation 84\/} (2018), 147--171.

\bibitem{Bejleritangentspacepunctual2017}
{\sc Bejleri, D., and Stapleton, D.}
\newblock The tangent space of the punctual {{Hilbert}} scheme.
\newblock {\em The Michigan Mathematical Journal 66}, 3 (Aug. 2017), 595--610.

\bibitem{blum_complexity_1998}
{\sc Blum, L., Cucker, F., Shub, M., and Smale, S.}
\newblock {\em Complexity and {{Real Computation}}}.
\newblock {Springer}, {NY}, 1998.

\bibitem{BrianconDescriptionHilb1977}
{\sc Brian{\c c}on, J.}
\newblock Description de {$Hilb^{n}C\{x,y\}$}.
\newblock {\em Inventiones mathematicae 41\/} (1977), 45--90.

\bibitem{BrianconDimensionpunctualHilbert1978}
{\sc Brian{\c c}on, J., and Iarrobino, A.}
\newblock Dimension of the punctual {{Hilbert}} scheme.
\newblock {\em Journal of Algebra 55}, 2 (Dec. 1978), 536--544.

\bibitem{DaytonLiZeng11}
{\sc Dayton, B.~H., Li, T.-Y., and Zeng, Z.}
\newblock Multiple zeros of nonlinear systems.
\newblock {\em Math. Comput. 80}, 276 (2011), 2143--2168.

\bibitem{zeng05}
{\sc Dayton, B.~H., and Zeng, Z.}
\newblock Computing the multiplicity structure in solving polynomial systems.
\newblock In {\em Proc. of ISSAC '05\/} (NY, USA, 2005), ACM, pp.~116--123.

\bibitem{DedieuShub2001}
{\sc Dedieu, J.-P., and Shub, M.}
\newblock On simple double zeros and badly conditioned zeros of analytic
  functions of {$n$} variables.
\newblock {\em Math. Comp. 70}, 233 (2001), 319--327.

\bibitem{DoubiletFoundationsCombinatorialTheory1974}
{\sc Doubilet, P., Rota, G.-C., and Stein, J.}
\newblock On the {{Foundations}} of {{Combinatorial Theory}}.
\newblock {\em Studies in Applied Mathematics 53}, 3 (1974), 185--216.

\bibitem{emiris:inria-00393833}
{\sc Emiris, I.~Z., Mourrain, B., and Tsigaridas, E.}
\newblock {The DMM bound: multivariate (aggregate) separation bounds}.
\newblock In {\em {Proceedings of the ISSAC'10}\/} (Munich, Germany, July
  2010), S.~Watt, Ed., {ACM}, pp.~243--250.

\bibitem{Giustietal2005}
{\sc Giusti, M., Lecerf, G., Salvy, B., and Yakoubsohn, J.-C.}
\newblock On location and approximation of clusters of zeros of analytic
  functions.
\newblock {\em Found. Comput. Math. 5}, 3 (2005), 257--311.

\bibitem{GLSY07}
{\sc Giusti, M., Lecerf, G., Salvy, B., and Yakoubsohn, J.-C.}
\newblock On location and approximation of clusters of zeros: Case of embedding
  dimension one.
\newblock {\em Foundations of Computational Mathematics 7\/} (2007), 1--58.

\bibitem{GiuYak13}
{\sc Giusti, M., and Yakoubsohn, J.-C.}
\newblock Multiplicity hunting and approximating multiple roots of polynomial
  systems.
\newblock vol.~604 of {\em Contemp. Math.}, AMS, pp.~105--128.

\bibitem{GiustiYak2018}
{\sc Giusti, M., and Yakoubsohn, J.-C.}
\newblock Approximation num{\'e}rique de racines isol{\'e}es multiples de
  syst{\`e}mes analytiques, 2018.
\newblock arXiv:1809.05446.

\bibitem{HaoSomZeng2013}
{\sc Hao, W., Sommese, A.~J., and Zeng, Z.}
\newblock Algorithm 931: an algorithm and software for computing multiplicity
  structures at zeros of nonlinear systems.
\newblock {\em ACM Trans. Math. Software 40}, 1 (2013), Art. 5, 16.

\bibitem{Hauensteindeflationmultiplicitystructure2016}
{\sc Hauenstein, J.~D., Mourrain, B., and Szanto, A.}
\newblock {On deflation and multiplicity structure}.
\newblock {\em {Journal of Symbolic Computation} 83\/} (2016), 228--253.

\bibitem{IarrobinoPunctualHilbertschemes1977}
{\sc Iarrobino, A.~A.}
\newblock {\em Punctual {{Hilbert}} Schemes}, vol.~188 of {\em Memoirs of the
  {{American Mathematical Society}}}.
\newblock {AMS}, {Providence}, 1977.

\bibitem{KanzawaOishi1997}
{\sc Kanzawa, Y., and Oishi, S.}
\newblock Approximate singular solutions of nonlinear equations and a numerical
  method of proving their existence.
\newblock No.~990. 1997, pp.~216--223.

\bibitem{LeeLiZhi2019}
{\sc Lee, K., Li, N., and Zhi, L.}
\newblock On isolation of singular zeros of multivariate analytic systems,
  2019.
\newblock arXiv:1904.0793.

\bibitem{lvz06}
{\sc Leykin, A., Verschelde, J., and Zhao, A.}
\newblock Newton's method with deflation for isolated singularities of
  polynomial systems.
\newblock {\em Theor. Computer Science 359}, 1-3 (2006), 111 -- 122.

\bibitem{lvz08}
{\sc Leykin, A., Verschelde, J., and Zhao, A.}
\newblock Higher-order deflation for polynomial systems with isolated singular
  solutions.
\newblock In {\em Algorithms in Algebraic Geometry}, A.~Dickenstein, F.-O.
  Schreyer, and A.~Sommese, Eds., vol.~146 of {\em The IMA Volumes in
  Mathematics and its Applications}. Springer, 2008, pp.~79--97.

\bibitem{LiZhi2013}
{\sc Li, N., and Zhi, L.}
\newblock Verified error bounds for isolated singular solutions of polynomial
  systems: case of breadth one.
\newblock {\em Theoret. Comput. Sci. 479\/} (2013), 163--173.

\bibitem{LiZhi2014}
{\sc Li, N., and Zhi, L.}
\newblock Verified error bounds for isolated singular solutions of polynomial
  systems.
\newblock {\em SIAM J. Numer. Anal. 52}, 4 (2014), 1623--1640.

\bibitem{LiSang2015}
{\sc Li, Z., and Sang, H.}
\newblock Verified error bounds for singular solutions of nonlinear systems.
\newblock {\em Numer. Algorithms 70}, 2 (2015), 309--331.

\bibitem{mm11}
{\sc Mantzaflaris, A., and Mourrain, B.}
\newblock Deflation and certified isolation of singular zeros of polynomial
  systems.
\newblock In {\em Proc. of ISSAC '11\/} (2011), ACM, pp.~249--256.

\bibitem{mm2014}
{\sc Mantzaflaris, A., and Mourrain, B.}
\newblock Singular zeros of polynomial systems.
\newblock In {\em Advances in Shapes, Geometry, and Algebra}, T.~Dokken and
  G.~Muntingh, Eds., vol.~10 of {\em Geometry and Computing}. Springer, 2014,
  pp.~77--103.

\bibitem{Mourrain97}
{\sc Mourrain, B.}
\newblock Isolated points, duality and residues.
\newblock {\em Journal of Pure and Applied Algebra 117-118\/} (1997), 469 --
  493.

\bibitem{RG10}
{\sc Rump, S., and Graillat, S.}
\newblock Verified error bounds for multiple roots of systems of nonlinear
  equations.
\newblock {\em Numerical Algorithms 54\/} (2010), 359--377.

\bibitem{ShafarevichBasicalgebraicgeometry2013}
{\sc Shafarevich, I.~R.}
\newblock {\em Basic Algebraic Geometry 1: Varieties in Projective Space}, 3rd
  edition~ed.
\newblock {Springer}, {New York}, 2013.

\bibitem{Wu:2008:CMS:1390768.1390812}
{\sc Wu, X., and Zhi, L.}
\newblock Computing the multiplicity structure from geometric involutive form.
\newblock In {\em Proceedings of ISSAC"08\/} (NY, USA, 2008), ACM,
  pp.~325--332.

\bibitem{WuZhi2011}
{\sc Wu, X., and Zhi, L.}
\newblock Determining singular solutions of polynomial systems via
  symbolic-numeric reduction to geometric involutive form.
\newblock {\em J. Symb. Comput. 27\/} (2008), 104--122.

\bibitem{Yakoubsohn2000}
{\sc Yakoubsohn, J.-C.}
\newblock Finding a cluster of zeros of univariate polynomials.
\newblock vol.~16. 2000, pp.~603--638.
\newblock Complexity theory, real machines, and homotopy (Oxford, 1999).

\bibitem{Yakoubsohn2002}
{\sc Yakoubsohn, J.-C.}
\newblock Simultaneous computation of all the zero-clusters of a univariate
  polynomial.
\newblock In {\em Foundations of computational mathematics ({H}ong {K}ong,
  2000)}. World Sci. Publ., River Edge, NJ, 2002, pp.~433--455.

\bibitem{Zeng2009}
{\sc Zeng, Z.}
\newblock The closedness subspace method for computing the multiplicity
  structure of a polynomial system.
\newblock vol.~496 of {\em Contemp. Math.} Amer. Math. Soc., Providence, RI,
  2009, pp.~347--362.

\end{thebibliography}

\bigskip
\end{document}